\documentclass[a4paper,11pt]{amsart}
\usepackage{graphicx,amssymb,amstext,amsmath}

\usepackage[latin1]{inputenc}
\usepackage{caption}
\usepackage{subcaption}
\usepackage{amsthm}
\usepackage{dsfont}
\usepackage{amsfonts}
\usepackage{amssymb}
\usepackage{mathrsfs}
\usepackage{graphicx,color}
\usepackage[all]{xy}
\usepackage[total={5.2in,9.1in},
top=1.4in, left=1.55in, includefoot]{geometry}

\newcommand{\ds}{\displaystyle}

\newcommand{\re}{\mathbb{R}}

\newtheorem{thm}{Theorem}[section]
\newtheorem*{thm*}{Theorem}
\newtheorem{cor}[thm]{Corollary}
\newtheorem{lem}[thm]{Lemma}

\theoremstyle{definition}

\theoremstyle{remark}
\newtheorem{rem}[thm]{Remark}

\numberwithin{equation}{section}

\newcommand{\p}{\mathcal P}
\newcommand{\h}{\mathcal H}

\newcommand{\la}{\lambda}

\title[A generalization of Marstrand's theorem and some geometric applications]{\bf{A generalization of Marstrand's theorem and some geometric applications}}%
\author{Carlos Gustavo Moreira,\\ Sergio Roma\~na, Waliston Luiz Silva}

\address{Carlos Gustavo~Moreira, Department of Mathematics \& SUSTech International Center for Mathematics, Southern University of Science and Technology, Shenzhen, Guangdong 518055, China and Instituto de Matem\'atica Pura e Aplicada (IMPA), Estrada Dona Castorina, 110, 22460-320 - Rio de Janeiro, RJ - Brazil}
\email{gugu@impa.br}
\address{Sergio Roma\~na, Instituto de Matem\'atica - Universidade Federal do Rio de Janeiro, Av. Pedro Calmon, 550 - Cidade Universit\'aria, 21941-901 - Rio de Janeiro, RJ, Brazil}
\email{sergiori@im.ufrj.br}
\address{Waliston Luiz Silva, Departamento de Matem\'atica e Estat\'istica, Universidade Federal de S\~ao Jo\~ao del Rei - Pra\c ca Frei Orlando 170, 36307-352 - S\~ao Jo\~ao del-Rei, MG, Brazil}
\email{waliston@yahoo.com}

\date{}
\begin{document}
\maketitle

\begin{abstract}
In this paper we prove using quite elementary methods, with a combinatorial nature, two general results related to Marstrand's projection theorem in a quite general formulation over metric spaces under a suitable transversality condition (the ``projections" are in principle only continuous, and the transversality condition gives flexibility in exponents) - the result is flexible enough to, in particular, recover most of the classical Marstrand-like theorems. We also give some new geometrical applications of our results - one of them is a new result related to Falconer's distance conjecture.
\end{abstract}

\maketitle
\section{Introduction}
Let $(Y,d)$ be a metric space. Given $\delta>0$ and $y\in Y$ we denote by $B(y,\delta)=\{z:d(z,y)<\delta\}$ the ball of center $y$ and radius $\delta$.\\
\indent If $U$ is a subset of $Y$, the diameter of $U$ is $|U|=\text{sup}\{d(x,y): x, y \in U\}$ and, if $\mathcal{U}$ is a family of subsets of $X$, the norm of $\mathcal{U}$ is defined as
$$\Vert\mathcal{U}\Vert=\ds \text{sup}_{U\in\mathcal{U}}|U|.$$
Given $\epsilon>0$ and $s>0$ denote $\ds\h^{s}_{\epsilon}(X)=\inf_{\substack{\mathcal{U} \ \text{covers} \ X \\ \Vert\mathcal{U}\Vert<\epsilon}}\,\sum_{U\in \mathcal{U}}|U|^s $. The Hausdorff $s$-measure is defined as 
$$\h^{s}(X)=\lim_{\epsilon\to 0}\h^{s}_{\epsilon}(X).$$
It is not difficult to show that there exists a unique $d\geq0$ for which $\h^{s}(X) =+\infty$  if $s < d$ and $\h^{s}(X) = 0$ if $s>d$. The Hausdorff dimension of $X$ is defined as $HD(X) = d$.
\ \\

Let $X$ be a subset of a complete metric space, $(\Lambda, \p)$ a probability space and $\pi:\Lambda\times X\to \re^k$ be a measurable map such that, for every $\lambda\in\Lambda$, the map $\pi_{\lambda}:X\to \re^k$ given by $\pi_{\lambda}(x):=\pi(\lambda,x)$ is continuous. Informally, one can think of $\pi_\la(\cdot)$ as a family of projections parameterized by $\la$. We assume that for some positives real numbers $\alpha, C$ the following transversality condition is satisfied:
\begin{equation}\label{transversality}
\p[\la\in \Lambda: d(\pi_{\la}(x_1),\pi_{\la}(x_2))\leq\delta d(x_1,x_2)^{\alpha}]\leq C\delta^{k}
\end{equation}
for all $\delta\in (0,1)$ and all $x_1, x_2 \in X$. 

Notice that, by replacing $C$ with max$\{1,C\}$, the above inequality, and also the inequality 
\begin{equation}\label{transversality2}
\p[\la\in \Lambda: d(\pi_{\la}(x_1),\pi_{\la}(x_2))\leq\delta d(x_1,x_2)^{\alpha}]\leq C\cdot\text{min}\{1,\delta\}^{k}
\end{equation} 
hold for every $\delta>0$.

\noindent We prove general versions of classical Marstrand's theorems stated originally for linear projections of plane sets on the real line (\cite{Marst}) assuming that $X$ is an analytic subset of a complete metric space (so it holds in particular if $X$ is a Borel set):

\begin{thm}\label{thm_Marstrand2} Let $X$ be an analytic subset of a complete metric space. 
\begin{enumerate}
\item[\emph{1.}] If $HD(X)>\alpha k$, then $Leb(\pi_{\lambda}(X))>0$ for almost every $\lambda\in \Lambda$. 
\item[\emph{2.}] If $HD(X)\leq \alpha k$, then $HD(\pi_{\lambda}(X))\geq \frac{HD(X)}{\alpha}$ for almost every $\lambda\in \Lambda$.
\end{enumerate}
\end{thm}

The following is an important result by Howroyd(see \cite{Howroyd_1995}):

\begin{lem}[Howroyd]\label{Howroyd}
Every analytic subset of a complete metric space with positive (or infinite) Hausdorff $s$-measure  contains a compact set which has finite and positive Hausdorff $s$-measure. 
\end{lem}

As we will see, it follows from Howroyd's result and from Lemma \ref{Frostman} that any set $X$ as in the previous statement satisfying $HD(X)>\alpha k$ contains a compact subset $X'$ such that $HD(X')=d>\alpha k$, $\h^d(X')>0$ and, for some constant $C>0$, $\h^d(X'\cap B(x, \delta))\leq C\delta^{d}$, for all $x\in X$ and all $\delta>0$. In the previous setting, for such a set $X'\subset X$, we have the following 

\begin{thm}\label{thm_L2} In the above conditions,
\begin{enumerate}
\item[\emph{1.}] For almost every $\lambda\in\Lambda$, the measure $\mu_{\lambda}:=(\pi_{\lambda})_*(\h^d|_{X'})$ is absolutely continuous with respect to the Lebesgue measure of ${\mathbb R}^k$.
\item[\emph{2.}] For almost every $\lambda\in\Lambda$, the Radon-Nykodim derivative $h_{\lambda}$ of $\mu_{\lambda}$ is a $L^2$ function, and $\int_{\Lambda}||h_{\lambda}||_{L^2}^2d\lambda<\infty$. 
\end{enumerate}
\end{thm}

The proofs of these results use quite elementary methods, with a combinatorial nature, and generalize the proof by the first named author and Y. Lima of the classical Marstrand's theorem for linear projections of plane sets on the real line (see \cite{YG}). 

In sections 3 and 4 we use these theorems to give a quick proof of a theorem by Solomyak (\cite{Solomyak}) on arithmetic sums of homogeneous Cantor sets, and we prove two new results: the first one is related to a famous conjecture by Falconer (see \cite{Fal_conj}), according to which, for every subset $E\subset \mathbb{R}^d$, $d\geq 2$, such that $HD(E)>d/2$, then the distance set
\[
\Delta(E) = \{|x-y|;\ x\in E, y\in E\},
\]
has positive Lebesgue measure, where $|\cdot|$ represents the Euclidean norm in $\mathbb{R}^d$. This conjecture is still open, and motivated several recent works, as \cite{Guth_et_al} and \cite{Du_et_al}.\\
Given any point $\lambda$ of $\mathbb{R}^d$ we define the set 

$$\Delta_{\lambda}(E):=\{|\lambda-y|: y\in E\},$$ 
 of distances from the point $\lambda$ to all points in the set $E$.
We prove that, if $HD(E)>1$ then, for almost all $\lambda\in \mathbb{R}^d$, the set $\Delta_{\lambda}(E)$ has positive Lebesgue measure.

The second one, proved in section 4, is a version of Marstrand's theorem in non-positive curvature.

We believe that our results will be also very useful in Dynamical Systems, for the study of the geometry of hyperbolic sets in arbitrary dimensions. In this context, conjugation maps and holonomies are often H\"older continuous, but not Lipschitz, and the flexibility given by the exponent $\alpha$ in theorems \ref{thm_Marstrand2} and \ref{thm_L2} will be an essential tool for new results related to the existence of blenders and to typical continuity of fractal dimensions of those sets.

{\bf Acknowledgement:} We would like to thank Jorge Erick L\'opez Vel\'azquez for helpful conversations on the subject of this paper.



\section{Proof of the Theorems}\label{section_results}
When $X$ is an analytic subset of a complete metric space and  $HD(X)>\alpha k$, by lemma \ref{Howroyd}, there is a compact subset of $\tilde{X}$ and $d>\alpha k$ such that $0<\h^d(\tilde{X})<+\infty$.\\
Thus, from now on,  we can assume $X$ is compact and satisfies $0<\h^d(X)<+\infty$ for some $d>\alpha k$. 
The following result is analogous to Frostman's Lemma (see for instance Theorem 8.8 of \cite{Mattila_1995}).

\begin{lem}\label{Frostman}
Given $\eta\in (0,1)$, there exists $X'\subset X$ compact such that $\h^d(X')>(1-\eta)\h^d(X)$ 
 and there is a constant $c$ such that 
\begin{equation}\label{eq1-Frostman}
\h^d(X'\cap B(x, \delta))\leq c\delta^{d},
\end{equation}
for all $x\in X$ and all $\delta>0$.
\end{lem}


\begin{proof}
By definition of Hausdorff dimension, there is $\epsilon>0$ such that $\h_{\epsilon}^{d}(X)>(1-\frac{\eta}2)\h^{d}(X)$.
Let $C_1:=2\cdot 5^d/\eta$ and consider the collection $\mathcal{B}$ of balls $B(x,\delta)$ such that $\h^d(X\cap B(x,\delta))>C_1\delta^{d}$ for some $x\in X$ and some positive $\delta\leq \frac{\epsilon}{5}$. \\
\textit{Claim:} The set $\ds \bigcup_{B\in \mathcal{B}}B$ satisfies 
\begin{equation}\label{eq2-Frostman}
\ds \h_{\epsilon}^{d}( \ds\bigcup_{B\in \mathcal{B}}B\cap X) <\frac{\eta}{2}\h^{d}(X).
\end{equation}
\textit{Proof of Claim:} Take a disjoint collection $B(x_1, \tilde{\delta}_1), B(x_2, \tilde{\delta}_2), \dots, B(x_n, \tilde{\delta}_n),\dots$ of elements of $\mathcal{B}$ such that  $\tilde{\delta}_{k}>\frac{1}{2}\sup \{\delta>0; \ \text{there is} \ B(x,\delta)\in\mathcal{B} \ \text{with} \ B(x,\delta)\cap \ds\bigcup_{j<k}B(x_j,\tilde{\delta}_{j}) =\emptyset \}$. Assume, without loss of generality, that the sequence $\{\tilde{\delta}_{i}\}_{i}$ is decreasing.\\
We have $\ds\bigcup_{i=1}^{\infty}B(x_i,5\tilde{\delta}_{i})\supset \ds\bigcup_{B\in \mathcal{B}}B$. Indeed, let $B(x,\delta)\in \mathcal{B}$ and consider $k_0=\max\{k:\tilde{\delta}_{k}\geq \frac{\delta}{2}\}$. If $B(x,\delta)=B(x_j,\tilde{\delta}_j)$ for some $j\le k_0+1$ we have nothing to do. If for all $j\leq k_0$, $B(x,\delta)\cap B(x_{j},\tilde{\delta}_{j})=\emptyset$, then by definition of $\tilde{\delta}_{k_{0}+1}$ we should have $\tilde{\delta}_{k_{0}+1}\geq \frac{\delta}{2}$ and this is a contradiction with the definition of $k_0$. Thus, there is $j\leq k_0$ such that $B(x,\delta)\cap B(x_{j},\tilde{\delta}_{j})\neq\emptyset$ and $\tilde{\delta}_{j}\geq \frac{\delta}{2}$. This last condition implies $B(x,\delta)\subset B(x_j, 5\tilde{\delta}_{j})$, and we are done.\\
However, by definition of $\mathcal{B}$, $\delta^d<\frac{\eta}{2\cdot 5^d}\h^{d}(X\cap B(x,\delta))$ for all $B(x,\delta)\in \mathcal{B}$. Therefore, 
\begin{equation}\label{eq3-Frostman}
\sum_{i=1}^{\infty}(5\tilde{\delta}_{i})^{d}=5^d\sum_{i=1}^{\infty}\tilde{\delta}_{i}^{d}<\frac{5^d\eta}{2\cdot 5^d}\sum_{i=1}^{\infty}\h^{d}(X\cap B(x_i,\tilde{\delta}_{i}))\leq \frac{\eta}{2}\h^{d}(X).
\end{equation}
Since $5\tilde{\delta}_{i}\leq \epsilon$, we concluded the proof of claim.\\
Now consider the set $\ds X'=X\setminus\bigcup_{B\in \mathcal{B}}B$ that satisfies 
\begin{eqnarray*}
\h^{d}(X')&\geq & \h^{d}_{\epsilon}(X') \\ 
&\ge&\h^{d}_{\epsilon}(X)-\h^{d}_{\epsilon}\left( X\cap\bigcup_{B\in \mathcal{B}}B\right)\\
&>& (1-\frac{\eta}2)\h^{d}(X)-\frac{\eta}{2}\h^{d}(X)=(1-\eta)\h^{d}(X).
\end{eqnarray*}
Moreover, by definition of $X'$ we have that $\h^{d}(X'\cap B(x,\delta))\leq C_1\delta^{d}$ for $\delta\leq \frac{\epsilon}{5}$.
Also, for $\delta>\frac{\epsilon}{5}$ we have $\h^{d}(X'\cap B(x,\delta))\le \h^{d}(X')\le \h^{d}(X)\le C_2\delta^{d}$, for $C_2:=(\frac{5}{\epsilon})^d\h^{d}(X)$.

Therefore, taking $c=\max\{\frac{2\cdot 5^d}{\eta} \, ,\left( \frac{5}{\epsilon}\right)^{d}\h^{d}(X) \}$ we conclude the proof.
\end{proof}
\ \\

\noindent Assume from now on that $X$ is compact and satisfies
$\h^d(X\cap B(x, \delta))\leq c\delta^{d}$,
for all $x\in X$ and all $\delta>0$.

\begin{lem}\label{Lemma2}
There is $\hat{C}>0$ such that for all $\eta>0$ there is a covering of $X$ by balls $B(x_i,\delta_i)$, $1\le i\le N$ with $\delta_i<\eta$ for all $i\in \{1,\dots, N\}$ with the following properties 
\begin{enumerate}
\item[\emph{1.}] $\ds\frac{3}{4}\h^{d}(X)<\sum_{i=1}^{N}\delta_i^{d}<\frac{5}{4}\h^{d}(X)$.
\item[\emph{2.}] For all $x\in X$ and $r>0$, we have $\ds \sum_{B(x_i,\delta_i)\subset B(x,r)}\delta^{d}_{i}<\hat{C}r^d$.
\end{enumerate}
\end{lem}
\begin{proof}
There is $\epsilon> 0$ such that $\h^d_{\epsilon}(X)>\frac{3}{4}\h^d(X)$. Let $\tilde\eta:=\min\{\epsilon,\eta\}$.

As $X$ is compact, then we start with a finite covering $\{B(z_i, \gamma_i), 1\le i\le m\}$ of $X$ with $\gamma_i<\tilde\eta$, $\forall i\le m$.

For $1\le j\le m$, let $A_j:=(X\cap B(z_j, \gamma_j))\setminus \bigcup_{i<j}B(z_i, \gamma_i)$. The sets $A_j$ are disjoint and cover $X$. We take, for each $j\le m$, a covering of $A_j$ by balls $B(y_{js},\epsilon_{js}), s\in {\mathbb N}$ (which effectively intersect $A_j$) with $\epsilon_{js}<\tilde\eta, \forall j, s$ and $\sum_{s\in {\mathbb N}}\epsilon_{js}^d\le \h^{d}(A_j)+\min\{\frac1{5m}\h^{d}(X),\frac{c}{5m}\tilde\eta^d\}, \forall j\le m$.

The balls $B(y_{js},\epsilon_{js})$ form a covering of $X$ from which we extract a finite subcovering (whose elements we denote provisionally by $B(x_i,\delta_i)$ - we will abuse of this notation: in what follows, we will modify the covering but we will keep the notation). If, for some $x\in X$ and some $r<\tilde\eta$, there are balls $B(x_i,\delta_i)$ from this subcovering contained in $B(x,r)$ for which $\sum \delta_i^d>r^d$, we replace all these balls $B(x_i,\delta_i)$ by $B(x,r)$. Doing this, the total number of balls in the covering and the sum of the $d$-th power of their radii decrease, and we obtain a new covering (whose elements we still denote by $B(x_i,\delta_i)$). Now, and until it is possible, we do the same procedure, which reduces the number of balls in the covering. After a finite number of steps, we arrive at a finite covering of $X$ by balls $B(x_i,\delta_i)$, $1\le i\le N$ with $\delta_i<\tilde\eta$ for all $i\in \{1,\dots, N\}$ such that, for every $x\in X$ and every $r<\tilde\eta$, the sum of the $d$-th power of the radii of the balls of the covering contained in $B(x,r)$ never exceeds $r^d$. By construction, we may associate to each such ball $B(x_i,\delta_i)$ a nonempty finite collection (perhaps with one element) of balls of the type $B(y_{js},\epsilon_{js})$ contained on it (perhaps $B(x_i,\delta_i)$ itself is of this form, otherwise we take the collection of balls of this form originally deleted in the process of creation of $B(x_i,\delta_i)$) with $\sum \epsilon_{js}^d\ge \delta_i^d$. Each ball of the type $B(y_{js},\epsilon_{js})$ belongs to at most one of these collections. Notice that, since $\h^d_{\epsilon}(X)>\frac{3}{4}\h^d(X)$, we have $\sum_{i=1}^{N}\delta_i^{d}>\frac{3}{4}\h^{d}(X)$.

Now, given $x\in X$ and $r\ge \tilde\eta$, if a ball $B(x_i,\delta_i)$ from the above covering is contained in $B(x,r)$, for each ball $B(y_{js},\epsilon_{js})$ from the collection associated to $B(x_i,\delta_i)$, we have $A_j\cap B(x,r)\ne\emptyset$, and thus, since $A_j\subset B(z_j, \gamma_j)$ and $\gamma_j<\tilde\eta\le r$, we have $A_j\subset B(x,3r)$. Therefore, since $\delta_i^d$ is at most the sum for the balls of its associated collection of $\epsilon_{js}^d$, and each ball of the type $B(y_{js},\epsilon_{js})$ belongs to at most one of these collections, the sum $\sum \delta_i^d$ for all such values of $i$ is at most 
$$\sum_{A_j\subset B(x,3r)}\sum_{s\in {\mathbb N}}\epsilon_{js}^d\le\sum_{A_j\subset B(x,3r)}(\h^{d}(A_j)+\frac{c}{5m}\tilde\eta^d)\le$$
$$\le\sum_{A_j\subset B(x,3r)}\h^{d}(A_j)+\frac{c}{5}\tilde\eta^d\le \h^{d}(X\cap B(x,3r))+\frac{c}{5}r^d\le \frac{6}{5}\cdot c\cdot (3r)^d=\hat C r^d,$$ 
with $\hat C:=\frac{6}{5}\cdot c\cdot 3^d$, since $\sum_{s\in {\mathbb N}}\epsilon_{js}^d\le \h^{d}(A_j)+\frac{c}{5m}\tilde\eta^d\}$ for all $j$.

Finally, we have 
$$\sum_{i=1}^{N}\delta_i^{d}\le \sum_{j\le m, s\in{\mathbb N}}\epsilon_{js}^d\le\sum_{j\le m}(\h^{d}(A_j)+\frac{1}{5m}\h^{d}(X))=$$	
$$=\sum_{j\le m}\h^{d}(A_j)+\frac{1}{5}\h^{d}(X)=\frac{6}{5}\h^{d}(X)<\frac{5}{4}\h^{d}(X).$$
\end{proof}

\begin{proof}[\emph{\textbf{Proof of Theorem \ref{thm_Marstrand2}}}]

(1) Take $\eta\in (0,1)$ small. Since $X$ is a compact set, we can consider a covering of $X$ by $N$ balls $B(x_i,\delta_i)$, $i=1,\dots,	N$ with $\delta_i\leq \eta$ for all $i$ as in Lemma \ref{Lemma2}. We define the following family of functions $f_{\lambda}\colon \re^k\to \re$, for $\lambda\in\Lambda$: 
$$f_{\lambda}(y)=\sum_{i=1}^{N}\delta_i^{d-\alpha k}\cdot\chi_{B(\pi_{\lambda}(x_i),\delta_i^{\alpha})}(y).$$
(Here we make an abuse of notation: it would be more precise to write $f_{\lambda}^{(\eta)}(y)$, or, even more correctly, $f_{\lambda}^{((B(x_i,\delta_i))_{1\le i\le N})}(y)$ instead of $f_{\lambda}(y)$, but we have chosen to make the notation simpler.)  

It is easy to see that $\text{supp}(f_{\lambda})\subset  V_{\eta^{\alpha}}\left( \pi_{\lambda}(X)\right)$, the neighborhood of radius $\eta^{\alpha}$ of $\pi_{\lambda}(X)$.

\noindent The idea now is to show that $\ds \int_{\Lambda}\int f_{\lambda}^{2}dyd\lambda<+\infty$, so we have that  $\ds \int f_{\lambda}^{2}dy<+\infty$ almost every $\lambda\in\Lambda$.
In fact: \\
\begin{equation*}\label{eq-pt}
\ds \int f_{\lambda}^{2}dy=\sum_{i=1}^{N}\delta_{i}^{2(d-\alpha k)}\int \chi_{B(\pi_{\lambda}(x_i),\delta_i^{\alpha})}(y)dy+2\sum_{i<j}(\delta_i\delta_j)^{d-\alpha k}\int \chi_{B(\pi_{\lambda}(x_i),\delta_i^{\alpha})\cap B(\pi_{\lambda}(x_j),\delta_j^{\alpha})}(y)dy.
\end{equation*}
The first part of the right side of the above inequality is bounded (which implies that its integral over $\lambda$ is also bounded). Indeed, if $w_k$ is the volume of the unit ball in $\re^k$, we have $\int \chi_{B(\pi_{\lambda}(x_i),\delta_i^{\alpha})}(y)dy=w_k\delta_i^{\alpha k}, \forall i\le N$, so $\sum_{i=1}^{N}\delta_{i}^{2(d-\alpha k)}\int \chi_{B(\pi_{\lambda}(x_i),\delta_i^{\alpha})}(y)dy=w_k\sum_{i=1}^{N}\delta_{i}^{2d-\alpha k}\le w_k\eta^{d-\alpha k}\sum_{i=1}^{N}\delta_{i}^{d}\le w_k\sum_{i=1}^{N}\delta_{i}^{d}\le \frac{5}{4}w_k\h^{d}(X)$ ( by (1) of Lemma \ref{Lemma2} and using $d>\alpha k$).

Let us now estimate the integral on $\lambda$ for the second part of the right side of the above inequality.\\ We use the notation $B_i:=B(\pi_{\lambda}(x_i),\delta_i^{\alpha})$, thus, assuming (by reordering, if necessary) that the sequence of the $\delta_{i}$'s is decreasing, we have 
\begin{eqnarray*}
\ds \int_{\Lambda}\, \, \sum_{i<j}(\delta_i\delta_j)^{d-\alpha k}\int \chi_{B_i(\lambda)\cap B_j(\lambda)}(y)dyd\lambda &\leq & w_k\sum_{i<j}(\delta_i\delta_j)^{d-\alpha k}\min\{\delta_i,\delta_j\}^{\alpha k}\cdot\p(\{\lambda: B_i\cap B_j\neq \emptyset \})\nonumber \\
&=&w_k\sum_{i<j}\delta_j^{d}\delta_i^{d-\alpha k}\cdot\p(\{\lambda: B_i\cap B_j\neq \emptyset \}).
\end{eqnarray*}
Note that by the transversality condition (\ref{transversality2}), and since $i<j$, we have that 
$$\p(\{\lambda: B_i\cap B_j\neq \emptyset \})\leq \p(\{\lambda: \pi_{\lambda}(x_j)\in B(\pi_{\lambda}(x_i),2\delta_i^{\alpha})\})\leq$$
$$\leq C\cdot \text{min}\{1,\frac{2\delta_i^{\alpha}}{d(x_i,x_j)^{\alpha}}\}\leq C\left(\frac{2\delta_i^{\alpha}}{\text{max}\{\delta_i,d(x_i,x_j)\}^\alpha} \right)^{k}.$$
Therefore, 
\begin{eqnarray}\label{eq1-pt}
\ds \int_{\Lambda}\, \, \sum_{i<j}(\delta_i\delta_j)^{d-\alpha k}\int \chi_{B_i(\lambda)\cap B_j(\lambda)}(y)dyd\lambda &\leq& C\cdot 2^k w_k \sum_{i<j}\frac{\delta_j^d \delta_i^{d-\alpha k}\delta_i^{\alpha k}}{\text{max}\{\delta_i,d(x_i,x_j)\}^{\alpha k}} \nonumber \\
&=& C\cdot 2^k w_k\sum_{i<j}\frac{\delta_i^d \delta_j^{d}}{\text{max}\{\delta_i,d(x_i,x_j)\}^{\alpha k}}.
\end{eqnarray}
The next step is to estimate the sum of the right side of the previous inequality. For this, we assume that $\text{diam}(X)\leq 2^{m_0}$ (for some $m_0\in\mathbb N$) and remember that $d>\alpha k$. Moreover, we have
\begin{eqnarray*}
\sum_{i<j}\frac{\delta_i^d \delta_j^{d}}{\text{max}\{\delta_i,d(x_i,x_j)\}^{\alpha k}}&=&\sum_{i=1}^{N}\delta_i^{d}\sum_{r=-m_0+1}^{\infty}\sum_{\substack{i<j ,\\2^{-r}<d(x_i,x_j)\leq 2^{1-r}}}\frac{ \delta_j^{d}}{\text{max}\{\delta_i,d(x_i,x_j)\}^{\alpha k}}\\
&\leq & \sum_{i=1}^{N}\delta_i^{d}\left(\sum_{r=-m_0+1}^{\lceil -\log\delta_i/\log 2\rceil}2^{\alpha k r}\sum_{\substack{i<j ,\\ d(x_i,x_j)\leq 2^{1-r}}}\delta_j^{d}+\delta_i^{-\alpha k}\sum_{\substack{i<j ,\\d(x_i,x_j)\leq \delta_i}}\delta_j^{d}\right).
\end{eqnarray*}
(indeed, if $r\ge \lceil -\log\delta_i/\log 2\rceil+1$ and $d(x_i,x_j)\leq 2^{1-r}$, we have $d(x_i,x_j)\le 2^{-\lceil -\log\delta_i/\log 2\rceil}\le\delta_i$, so $\text{max}\{\delta_i,d(x_i,x_j)\}=\delta_i$.)
Since $i<j$ and $d(x_i,x_j)\leq \delta_i$ imply $B(x_j,\delta_j)\subset B(x_i,2\delta_i)$, by (2) of Lemma \ref{Lemma2} we have $\displaystyle\sum_{\substack{i<j ,\\d(x_i,x_j)\leq \delta_i}}\delta_j^{d}<\hat C(2\delta_i)^d$, and thus 
$$\sum_{i=1}^{N}\delta_i^{d-\alpha k}\sum_{\substack{i<j ,\\d(x_i,x_j)\leq \delta_i}}\delta_j^{d}\le 2^d\hat C\sum_{i=1}^{N}\delta_i^{2d-\alpha k}\le 2^d\hat C\eta^{d-\alpha k}\sum_{i=1}^{N}\delta_{i}^{d}\le 2^d\hat C\sum_{i=1}^{N}\delta_{i}^{d}\le \frac{5}{4}2^d\hat C\h^{d}(X)$$ (by (1) of Lemma \ref{Lemma2}).
On the other hand, again by (2) and (1) of Lemma \ref{Lemma2}, we have 
$$\sum_{i=1}^{N}\delta_i^{d}\sum_{r=-m_0+1}^{\lceil -\log\delta_i/\log 2\rceil}2^{\alpha k r}\sum_{\substack{i<j ,\\ d(x_i,x_j)\leq 2^{1-r}}}\delta_j^{d}\leq \sum_{i=1}^{N}\delta_i^{d}\sum_{r=-m_0+1}^{\lceil -\log\delta_i/\log 2\rceil}2^{\alpha k r}\hat{C}(2^{2-r})^{d}\leq $$
$$\leq \frac{5}{4}\h^{d}(X)\hat{C}2^{2d}\sum_{r=-m_0+1}^{\infty}2^{-r(d-\alpha k )}\leq \frac{5}{4}\hat{C}2^{2d}\h^{d}(X)\frac{2^{(m_0-1)(d-\alpha k)}}{1-2^{\alpha k-d}}<+\infty.$$

Thus, we conclude that $\ds \int_{\Lambda}\int f_{\lambda}^{2}dyd\lambda\leq B<+\infty$ (notice that $B$ does not depend on the covering of X by the balls $B(x_j,\delta_j)$). It follows that, for every $\tau>0$, $\mathcal{P}(\{\lambda:\int f_{\lambda}^{2}dy>\tau^{-1}\})\leq B\tau$.\\
\ \\
Given a positive integer $n$, take, in the beginning of the proof, $\eta=\frac{1}{n}$ in the definition of $f_{\lambda}$, which we now denote by $f_{\lambda}^{[n]}$. For every $\tau>0$, there is a measurable set $\Lambda_n(\tau)\subset \Lambda$ such that $\mathcal{P}(\Lambda_n(\tau))\geq 1-B\tau$ and for $\lambda\in \Lambda_n(\tau)$ we have $\int (f_{\lambda}^{[n]})^{2}dy\leq \tau^{-1}$. Therefore we have $\mathcal{P}(\limsup \Lambda_n(\tau))\geq 1-B\tau$, where $\limsup \Lambda_n(\tau)=\{\lambda:\lambda\in \Lambda_n(\tau) \ \text{for infinitely many indices} \ n\}$, and, given $\lambda\in\limsup \Lambda_n(\tau)$, we have $\int (f_{\lambda}^{[n]})^{2}dy\leq \tau^{-1}$ for infinitely many values of $n$. \\
\noindent By the Cauchy-Schwarz inequality we have 
$$\ds\left( \int f_{\lambda}^{[n]}(y)dy\right)^{2}=\left( \int f_{\lambda}^{[n]}\cdot \chi_{\text{supp}(f_{\lambda}^{[n]})}dy\right)^{2}\leq Leb(\text{supp}(f_{\lambda}^{[n]})) \cdot\int (f_{\lambda}^{[n]})^{2}dy.$$
However, the first condition of Lemma \ref{Lemma2} gives us $$\ds \int f_{\lambda}^{[n]}dy=w_k\sum_{i}^{N}\delta_{i}^{d-\alpha k}\delta_{i}^{\alpha k}=w_k\sum_{i=1}^{N}\delta_{i}^{d}> \frac{3}{4}w_k\h^{d}(X)=:\tilde{B}>0,$$
 and therefore we get that if 
$\lambda \in \limsup \Lambda_n(\tau)$, then $Leb(\text{supp}(f_{\lambda}^{[n]}))\geq \tilde{B}^2\tau>0$ for infinitely many indices $n$, which implies that $Leb(\pi_\lambda(X))>0$. Finally, taking $\tau=\frac{1}{m}$, we have that 
$$\ds\mathcal{P}\left(\bigcup_{m\geq 1}\limsup \Lambda_n(1/m)\right)=1$$ 
and $Leb(\pi_\lambda(X))>0$ for any $\ds\lambda\in \bigcup_{m\geq 1}\limsup \Lambda_n(1/m)$. 

Let us remark that the above proof can be used to show that, given a subset $Y\subset X$ with $\h^d(Y)>0$, we also have $Leb(\pi_\lambda(Y))>0$ for almost all $\lambda\in\Lambda$. Indeed, by Lemma \ref{Howroyd}, me may assume that $Y$ is compact.  We may define $g_{\lambda}\colon \re^k\to \re$: 
$$g_{\lambda}(y)=\sum_{\substack{1\le i\le N \\ B(\pi_{\lambda}(x_i),\delta_i)\cap Y\ne \emptyset}}\delta_i^{d-\alpha k}\cdot\chi_{B(\pi_{\lambda}(x_i),\delta_i^{\alpha})}(y).$$
We have $g_{\lambda}(y)\le f_{\lambda}(y), \forall y\in \re^k$, so $\ds \int_{\Lambda}\int g_{\lambda}^{2}dyd\lambda\leq B<+\infty$.
Again by the Cauchy-Schwarz inequality we have 
$$\ds\left(  \int g_{\lambda}dy\right)^{2}\leq Leb(\text{supp}(g_{\lambda})) \cdot\int g_{\lambda}^{2}dy,$$
and we also have, as before, 
$$\ds \int g_{\lambda}dy=w_k\sum_{\substack{1\le i\le N \\ B(\pi_{\lambda}(x_i),\delta_i)\cap Y\ne \emptyset}}\delta_{i}^{d-\alpha k}\delta_{i}^{\alpha k}=w_k\sum_{\substack{1\le i\le N \\ B(\pi_{\lambda}(x_i),\delta_i)\cap Y\ne \emptyset}}\delta_{i}^{d}> \frac{3}{4}w_k\h^{d}(Y)>0,$$
provided $\eta$ is small enough. So we conclude as before that $Leb(\pi_\lambda(Y))>0$ for any $\ds\lambda\in \bigcup_{m\geq 1}\limsup \Lambda_n(1/m)$. 

This finishes the proof of the first part of the theorem. \\
\ \\

(2)
We will prove that, for every $\varepsilon>0$, there is a full measure subset $\Lambda_{\varepsilon}$ of $\Lambda$ such that, for every $\lambda\in \Lambda_{\varepsilon}$, $HD(\pi_{\lambda}(X))>\frac{HD(X)}{\alpha}-\varepsilon$. It follows that $\Lambda':=\cap_{m=1}^{\infty}\Lambda_{1/m}$ is a full measure subset of $\Lambda$ such that, for every $\lambda\in \Lambda'$, $HD(\pi_{\lambda}(X))\ge \frac{HD(X)}{\alpha}$.

\noindent For $0<\beta<\frac{1}{2}$ consider the homogeneous Cantor set given by
$$K_{\beta}=\{(1-\beta)\sum_{k=0}^{\infty}a_k \beta^k; a_k\in\{0,1\}, \forall k\ge 0\},$$
whose Hausdorff dimension coincides with its (upper) box dimension, and is given by $\frac{\log 2}{|\log \beta|}$. \\

Choose $0<\beta<\frac{1}{2}$ and a positive integer $r$ such that $\beta^{\alpha}<\frac{1}{2}$ and $k<\frac1{\alpha}(HD(X)+r\cdot\frac{\log 2}{|\log \beta|})<k+\varepsilon$. Consider the homeomorphism $g:K_{\beta}\to K_{\beta^{\alpha}}$ given by 
$$g((1-\beta)\sum_{k=0}^{\infty}a_k \beta^k)=(1-\beta^{\alpha})\sum_{k=0}^{\infty}a_k \beta^{\alpha k}; a_k\in\{0,1\}, \forall k\ge 0.$$
There are constants $a, b>0$ such that $a|x-y|^{\alpha}\le |g(x)-g(y)|\le b|x-y|^{\alpha}, \forall x, y\in K_{\beta}$.

Let $B=\{v\in {\mathbb R}^k; |v|\le 1\}$ (provided with the normalized Lebesgue measure), $\tilde\Lambda=\Lambda\times B^r$ and $\tilde\pi:\tilde\Lambda\times X\times K_{\beta}^r\to {\mathbb R}^k$ as follows: given $\tilde\lambda=(\lambda,v_1,v_2,\dots, v_r)\in\tilde\Lambda$ and $(x,t_1,t_2,\dots,t_r)\in X\times K_{\beta}^r$, we define $\tilde\pi_{\tilde\lambda}(x,t_1,t_2,\dots,t_r)=\tilde\pi(\tilde\lambda,x,t_1,t_2,\dots,t_r):=\pi_{\lambda}(x)+\sum_{j=1}^r g(t_j)v_j$. We consider in $X\times K_{\beta}^r$ the distance $\tilde d((x,t_1,t_2,\dots,t_r),(x',t_1',t_2',\dots,t_r')):=\text{max}\{d(x,x'),|t_j-t_j'|, 1\le j\le r\}$. We will prove that $\tilde\pi$ satisfies the transversality condition (for the exponent $\alpha$ and a suitable positive constant $\tilde C$), and therefore, since $HD(X\times K_{\beta}^r)=HD(X)+r\frac{\log 2}{|\log \beta|}>\alpha k$, 
we conclude, using the case (1) of the theorem that, for almost every $\tilde\lambda=(\lambda,v_1,v_2,\dots, v_r)\in\tilde\Lambda$, $\tilde\pi_{\tilde\lambda}(X\times K_{\beta}^r)$ has positive Lebesgue measure in ${\mathbb R}^k$ (and thus Hausdorff dimension $k$). In particular, for almost every $\lambda\in\Lambda$, there is $(v_1,v_2,\dots, v_r)\in B^r$ such that, if $H(s_1,s_2,\dots,s_r):=\sum_{j=1}^r s_j v_j$, $\pi_{\lambda}(X)+H(K_{\beta^{\alpha}}^r)=\tilde\pi_{(\lambda,v_1,v_2,\dots,v_r)}(X\times K_{\beta}^r)$ has Hausdorff dimension $k$. 
Notice that $H$ is a Lipschitz map (so the map $\tilde H:{\mathbb R}^k\times K_{\beta^{\alpha}}^r\to {\mathbb R}^k$ given by $\tilde H(y,s_1,s_2,\dots,s_r):=y+H(s_1,s_2,\dots,s_r)$ is also Lipschitz), and $\tilde H(\pi_{\lambda}(X)\times K_{\beta^{\alpha}}^r)=\tilde\pi_{(\lambda,v_1,v_2,\dots,v_r)}(X\times K_{\beta}^r)$ has Hausdorff dimension $k$. This implies that the cartesian product $\pi_{\lambda}(X)\times K_{\beta^{\alpha}}^r$ has Hausdorff dimension at least $k$. Since the box dimension of $K_{\beta^{\alpha}}^r$ is $\frac{r}{\alpha}\cdot \frac{\log 2}{|\log \beta|}$, it follows that $HD(\pi_{\lambda}(X))\ge k-\frac{r}{\alpha}\cdot \frac{\log 2}{|\log \beta|}>\frac{HD(X)}{\alpha}-\varepsilon$, which will conclude the proof. 

Let us now prove the transversality condition. Given $y=(x,t_1,t_2,\dots,t_r),y'=(x',t_1',t_2',\dots,t_r')\in X\times K_{\beta}^r$ and $\delta\in (0,1)$, we have 3 cases:

(i) $\text{max}\{|t_j-t_j'|, 1\le j\le r\}\le \delta^{1/\alpha}d(x,x')$.

In this case, $|\sum_{j=1}^r g(t_j)v_j-\sum_{j=1}^r g(t_j')v_j|\le r\cdot \text{max}\{|g(t_j)-g(t_j')|, 1\le j\le r\}\le rb\delta\cdot d(x,x')^{\alpha}$. On the other hand, by the transversality condition, the probability in $\Lambda$ that $d(\pi_{\lambda}(x),\pi_{\lambda}(x'))\le (1+rb)\delta\cdot d(x,x')^{\alpha}$ is at most $C((1+rb)\delta)^k=C(1+rb)^k\delta^k$. If $d(\pi_{\lambda}(x),\pi_{\lambda}(x'))>(1+rb)\delta\cdot d(x,x')^{\alpha}$, then 
$$|\pi_{\lambda}(x)+\sum_{j=1}^r g(t_j)v_j-(\pi_{\lambda}(x')+\sum_{j=1}^r g(t_j')v_j)|>(1+rb)\delta\cdot d(x,x')^{\alpha}-rb\delta\cdot d(x,x')^{\alpha}=\delta\cdot d(x,x')^{\alpha}.$$
As a consequence, since $d(y,y')=d(x,x')$ in this case, we have 
$$\p[\tilde\la\in \tilde\Lambda: d(\pi_{\tilde\la}(y),\pi_{\la}(y'))\leq\delta d(y,y')^{\alpha}]=\p[\tilde\la\in \tilde\Lambda: d(\pi_{\tilde\la}(x),\pi_{\la}(x'))\leq\delta d(y,y')^{\alpha}]\leq (1+rb)^k\delta^{k}.$$

(ii) $\text{max}\{|t_j-t_j'|, 1\le j\le r\}\ge d(x,x')$.

In this case, fix $i\le r$ such that $|t_i-t_i'|=\text{max}\{|t_j-t_j'|, 1\le j\le r\}\ge d(x,x')$. Therefore $|t_i-t_i'|=\tilde d((x,t_1,t_2,\dots,t_r),(x',t_1',t_2',\dots,t_r'))$ (and $|g(t_i)-g(t_i')|\ge a\cdot |t_i-t_i'|^{\alpha}$). Given \break $\lambda\in\Lambda, v_1, v_2,\dots,v_{i-1},v_{i+1},\dots,v_r\in B$, we have $|\pi_{\lambda}(x)+\sum_{j=1}^r g(t_j)v_j-(\pi_{\lambda}(x')+\sum_{j=1}^r g(t_j')v_j)|\le \delta|t_i-t_i'|^{\alpha}$ if and only if $|v_i-u|\le \delta|t_i-t_i'|^{\alpha}/|g(t_i)-g(t_i')|$, where \hfill\break $u:=(\sum_{1\le j\le r, j\ne i}(g(t_j')-g(t_j))v_j+\pi_{\lambda}(x')-\pi_{\lambda}(x))/(g(t_i)-g(t_i'))$ (this is only a rephrasing). The probability (in $B$) that this happens is
$\text{vol}(B(u,\delta|t_i-t_i'|^{\alpha}/|g(t_i)-g(t_i')|)\cap B)/\text{vol}(B)\le$ \hfill\break 
$\le\text{vol}(B(u,\delta|t_i-t_i'|^{\alpha}/|g(t_i)-g(t_i')|))/\text{vol}(B)=(\delta|t_i-t_i'|^{\alpha}/|g(t_i)-g(t_i')|)^k\le (\delta/a)^k$. By Fubini's theorem, the probability in $\tilde\Lambda$ that $|\pi_{\lambda}(x)+\sum_{j=1}^r g(t_j)v_j-(\pi_{\lambda}(x')+\sum_{j=1}^r g(t_j')v_j)|\le \delta\tilde d((x,t_1,t_2,\dots,t_r),(x',t_1',t_2',\dots,t_r'))^{\alpha}=\delta|t_i-t_i'|^{\alpha}$ is at most $(\delta/a)^k$.

(iii) $\delta^{1/\alpha}d(x,x')<\text{max}\{|t_j-t_j'|, 1\le j\le r\}<d(x,x')$.

In this case, we have $\tilde d((x,t_1,t_2,\dots,t_r),(x',t_1',t_2',\dots,t_r'))=d(x,x')$. Let $s\in\mathbb N$ such that $2^{-s-1}d(x,x')<\text{max}(|t_j-t_j'|, 1\le j\le r)\le 2^{-s}d(x,x')$. We have $2^{-s}>\delta^{1/\alpha}$ and $|\sum_{j=1}^r g(t_j)v_j-\sum_{j=1}^r g(t_j')v_j|\le r\cdot \text{max}(|g(t_j)-g(t_j')|, 1\le j\le r)\le rb2^{-s\alpha}\cdot d(x,x')^{\alpha}$. The probability in $\Lambda$ that $d(\pi_{\lambda}(x),\pi_{\lambda}(x'))\le (1+rb)2^{-s\alpha}\cdot d(x,x')^{\alpha}$ is at most $C((1+rb)2^{-s\alpha})^k=C(1+rb)^k2^{-sk\alpha}$. If $d(\pi_{\lambda}(x),\pi_{\lambda}(x'))>(1+rb)2^{-s\alpha}\cdot d(x,x')^{\alpha}$, then 
$|\pi_{\lambda}(x)+\sum_{j=1}^r g(t_j)v_j-(\pi_{\lambda}(x')+\sum_{j=1}^r g(t_j')v_j)|>(1+rb)2^{-s\alpha}\cdot d(x,x')^{\alpha}-rb2^{-s\alpha}\cdot d(x,x')^{\alpha}=2^{-s\alpha}\cdot d(x,x')^{\alpha}>\delta\cdot d(x,x')^{\alpha}$.

Therefore, if we have $|\pi_{\lambda}(x)+\sum_{j=1}^r g(t_j)v_j-(\pi_{\lambda}(x')+\sum_{j=1}^r g(t_j')v_j)|\le\delta\cdot d(x,x')^{\alpha}$, we should have $d(\pi_{\lambda}(x),\pi_{\lambda}(x'))\le (1+rb)2^{-s\alpha}\cdot d(x,x')^{\alpha}$. Fix $i\le r$ such that $|t_i-t_i'|=\text{max}(|t_j-t_j'|, 1\le j\le r)>2^{-s-1}d(x,x')$. We have $|g(t_i)-g(t_i')|\ge a\cdot |t_i-t_i'|^{\alpha}>a2^{-(s+1)\alpha}d(x,x')^{\alpha}$. Given $\lambda\in\Lambda, v_1, v_2,\dots,v_{i-1},v_{i+1},\dots,v_r\in B$, we have $|\pi_{\lambda}(x)+\sum_{j=1}^r g(t_j)v_j-(\pi_{\lambda}(x')+\sum_{j=1}^r g(t_j')v_j)|\le \delta d(x,x')^{\alpha}$ if and only if $|v_i-u|\le \delta d(x,x')^{\alpha}/|g(t_i)-g(t_i')|$, where $u:=(\sum_{1\le j\le r, j\ne i}(g(t_j')-g(t_j))v_j+\pi_{\lambda}(x')-\pi_{\lambda}(x))/(g(t_i)-g(t_i'))$. The probability (in $B$) that this happens is
$\text{vol}(B(u,\delta d(x,x')^{\alpha}/|g(t_i)-g(t_i')|)\cap B)/\text{vol}(B)\le$ \hfill\break 
$\le\text{vol}(B(u,\delta d(x,x')^{\alpha}/|g(t_i)-g(t_i')|))/\text{vol}(B)=(\delta d(x,x')^{\alpha}/|g(t_i)-g(t_i')|)^k\le (2^{(s+1)\alpha}\delta/a)^k$. So, the probability in $\Lambda\times B$ that, simultaneously, $d(\pi_{\lambda}(x),\pi_{\lambda}(x'))\le (1+rb)2^{-s\alpha}\cdot d(x,x')^{\alpha}$ and $|v_i-u|\le \delta d(x,x')^{\alpha}/|g(t_i)-g(t_i')|$ is at most $C(1+rb)^k2^{-sk\alpha}\times (2^{(s+1)\alpha}\delta/a)^k=C((1+rb)/a)^k2^{k\alpha}\delta^k$. By Fubini's theorem, the probability in $\tilde\Lambda$ that $|\pi_{\lambda}(x)+\sum_{j=1}^r g(t_j)v_j-(\pi_{\lambda}(x')+\sum_{j=1}^r g(t_j')v_j)|\le \delta\tilde d(((x,t_1,t_2,\dots,t_r),(x',t_1',t_2',\dots,t_r'))^{\alpha}$ is at most $C((1+rb)/a)^k2^{k\alpha}\delta^k$.

This concludes the proof, with $\tilde C:=\text{max}\{C(1+rb)^k, 1/a^k, C((1+rb)/a)^k2^{k\alpha}\}$.
\end{proof}

\begin{cor}
Under the conditions of theorem \ref{thm_Marstrand2}, part \emph{(2)}, if $\pi_{\lambda}$ is $\alpha$-H\"older for every $\lambda\in\Lambda$, then $HD(\pi_{\lambda}(X))=\frac{HD(X)}{\alpha}$ for almost every $\lambda\in\Lambda$.
\end{cor}

\begin{proof}[\emph{\textbf{Proof of Theorem \ref{thm_L2}}}]

(1) In the remark at the end of the proof of part one of Theorem \ref{thm_Marstrand2}, we proved that there is a full measure subset $\Lambda'\subset\Lambda$ such that if $Y\subset X$ and $\h^d(Y)>0$, then $Leb(\pi_{\lambda}(Y))>0$ for every $\lambda\in \Lambda'$. Since $\mu_{\lambda}=(\pi_{\lambda})_*(\h^d|_{X'})$, given $A\subset \re^k$ with $Leb(A)=0$, we have $\mu_{\lambda}(A)=\h^d(\pi_{\lambda}^{-1}(A)\cap X')=0$, otherwise $Leb(A)=Leb(\pi_{\lambda}(\pi_{\lambda}^{-1}(A)\cap X'))>0$, a contradiction. It follows that the measure $\mu_{\lambda}$ is absolutely continuous with respect to the Lebesgue measure of ${\mathbb R}^k$.

(2) Given $\lambda\in\Lambda'$, let $h_{\lambda}=d\mu_{\lambda}/dLeb$. We define, for every $\varepsilon>0$, $h_{\lambda,\varepsilon}:\re^k\to\re$ by
$$h_{\lambda,\varepsilon}(x)=\frac1{w_k \varepsilon^k}\int_{B(x,\varepsilon)}h_{\lambda}(y)dy.$$
Since $h_{\lambda}$ is in $L^1$, the Lebesgue differentiation theorem implies that $h_{\lambda}(x)=\displaystyle\lim_{\varepsilon\to 0}h_{\lambda,\varepsilon}(x)$ for almost every $x\in\re^k$.

Using the estimates of the previous theorem, we will prove that, for almost every $\lambda\in \Lambda$, $||h_{\lambda,\varepsilon}||_{L^2}$ is bounded by a constant that does not depend on $\varepsilon$, and so it follows from Fatou's lemma that, for almost every $\lambda\in \Lambda$, $h_{\lambda}$ is $L^2$ (indeed $||h_{\lambda}||_{L^2}$ will be bounded by the same constant).

In order to do this, notice that 
$$h_{\lambda,\varepsilon}(x)=\frac1{w_k \varepsilon^k}\int_{B(x,\varepsilon)}h_{\lambda}(y)dy=\frac1{w_k \varepsilon^k}\h^d(\pi_{\lambda}^{-1}(B(x,\varepsilon))\cap X').$$
In order to estimate the last term, for every positive integer $n$, consider a covering of $X'$ by balls $B(x_j,\delta_j)$ as in the end of the proof of case (1) of the previous theorem, with norm smaller than $1/n$. If $1/n<\varepsilon$, we have:
\begin{eqnarray*}
\h^d(\pi_{\lambda}^{-1}(B(x,\varepsilon))\cap X')&\le & \sum_{\substack{1\le j\le N,\\ \pi_{\lambda}(B(x_j,\delta_j))\subset B(x,2\varepsilon)}}\h^d(B(x_j,\delta_j)\cap X')\\
&\le & \sum_{\substack{1\le j\le N,\\ \pi_{\lambda}(B(x_j,\delta_j))\subset B(x,2\varepsilon)}}c(2\delta_j)^d\\
&=&\frac{c\cdot 2^d}{w_k}\sum_{\substack{1\le j\le N,\\ \pi_{\lambda}(B(x_j,\delta_j))\subset B(x,2\varepsilon)}}\delta_j^{d-\alpha k}Leb(B(\pi_{\lambda}(x_j),\delta_j^{\alpha})\\
&\le &\frac{c\cdot 2^d}{w_k}\int_{B(x,2\varepsilon)}f_{\lambda}^{(n)}(y)dy,
\end{eqnarray*}
where $f_{\lambda}^{(n)}(y)=\sum_{i=1}^{N}\delta_i^{d-\alpha k}\cdot\chi_{B(\pi_{\lambda}(x_i),\delta_i^{\alpha})}(y)$, as in the previous theorem.


\noindent By the Cauchy-Schwarz inequality, we have 
$$|h_{\lambda,\varepsilon}(x)|^2=\frac1{(w_k \varepsilon^k)^2}\h^d(\pi_{\lambda}^{-1}(B(x,\varepsilon))\cap X')^2\le \frac{c^2\cdot 2^{2d}}{w_k^4\varepsilon^{2k}}\left(\int_{B(x,2\varepsilon)}f_{\lambda}^{(n)}(y)dy\right)^2\le$$
$$\le \frac{c^2\cdot 2^{2d}}{w_k^4\varepsilon^{2k}}\cdot w_k(2\varepsilon)^k\int_{B(x,2\varepsilon)}f_{\lambda}^{(n)}(y)^2dy=\frac{c^2\cdot 2^{2d+k}}{w_k^3\varepsilon^{k}}\int_{B(x,2\varepsilon)}(f_{\lambda}^{(n)}(y))^2dy.$$

It follows that
$$||h_{\lambda,\varepsilon}||_{L^2}^2=\int|h_{\lambda,\varepsilon}(x)|^2dx\le \frac{c^2\cdot 2^{2d+k}}{w_k^3\varepsilon^{k}}\int\int_{B(x,2\varepsilon)}(f_{\lambda}^{(n)}(y))^2dydx=$$
$$=\frac{c^2\cdot 2^{2d+k}}{w_k^3\varepsilon^{k}}\int\int(f_{\lambda}^{(n)}(y))^2\chi_{B(x,2\varepsilon)}dydx=\frac{c^2\cdot 2^{2d+k}}{w_k^3\varepsilon^{k}}w_k(2\varepsilon)^k\int f_{\lambda}^{(n)}(y)^2dy,$$
which is equal to $\frac{c^2\cdot 2^{2d+2k}}{w_k^2}||f_{\lambda}^{(n)}||_{L^2}^2$. 
In the end of the proof of case (1) of the previous theorem, we showed that, for every $\tau>0$, there is a set $\Lambda(\tau):=\limsup \Lambda_n(\tau)$ such that $\mathcal{P}(\Lambda(\tau))\geq 1-B\tau$ and, for $\lambda\in\Lambda(\tau)$, there are infinitely many positive integers $n$ for which $||f_{\lambda}^{(n)}||_{L^2}^2=\int f_{\lambda}^{(n)}(y)^2dy\le \tau^{-1}$. It follows from the previous inequality that, for $\lambda\in\Lambda(\tau)$, $||h_{\lambda,\varepsilon}||_{L^2}^2\le \frac{c^2\cdot 2^{2d+2k}}{w_k^2\tau}$.

On the other hand, we have $\int_{\Lambda}||f_{\lambda}^{(n)}||_{L^2}^2d\lambda=\int_{\Lambda}\int (f_{\lambda}^{(n)}(y))^{2}dyd\lambda\leq B$ for every positive integer $n$, so, for almost every $\lambda\in\Lambda$ (more especifically, for every $\lambda\in\cup_{m\ge 1}\Lambda(1/m)$), we have $$\int||h_{\lambda,\varepsilon}||_{L^2}^2d\lambda\le\frac{c^2\cdot 2^{2d+2k}}{w_k^2}\int||f_{\lambda}^{(n)}||_{L^2}^2d\lambda\le\frac{Bc^2\cdot 2^{2d+2k}}{w_k^2},$$
from which we conclude, again by Fatou's lemma, that 
$$\int_{\Lambda}||h_{\lambda}||_{L^2}^2d\lambda\le\liminf_{\varepsilon\to 0}\int||h_{\lambda,\varepsilon}||_{L^2}^2d\lambda\le\frac{Bc^2\cdot 2^{2d+2k}}{w_k^2}<\infty.$$
\end{proof}

\begin{cor}
Under the conditions of theorem \ref{thm_Marstrand2}, part (1), we have 
$$\int_{\Lambda}(Leb(\pi_{\lambda}(X))^{-1}d\lambda<\infty.$$
\end{cor}

\begin{proof}
Given $X$ as in theorem \ref{thm_Marstrand2}, part (1). Take a compact $X'\subset X$ such that $HD(X')=d>\alpha k$, $\h^d(X')>0$ and, for some constant $c>0$, $\h^d(X'\cap B(x, \delta))\leq c\delta^{d}$, for all $x\in X$ and all $\delta>0$. By theorem \ref{thm_L2}, for almost every $\lambda\in\Lambda$, for $\mu_{\lambda}=(\pi_{\lambda})_*(\h^d|_{X'})$, $h_{\lambda}=d\mu_{\lambda}/dLeb$ is $L^2$. For these values of $\lambda$, Cauchy-Scwarz inequality gives
$$\h^d(X')^2=\mu_{\lambda}(\re^k)^2=\left(\int h_{\lambda}(y)dy\right)^2=\left(\int h_{\lambda}(y)\chi_{\pi_{\lambda}(X')}dy\right)^2\le Leb(\pi_{\lambda}(X'))||h_{\lambda}||_{L^2}^2,$$
and thus $Leb(\pi_{\lambda}(X))^{-1}\le Leb(\pi_{\lambda}(X'))^{-1}\le \h^d(X')^{-2}||h_{\lambda}||_{L^2}^2$, so 
$$\int_{\Lambda}(Leb(\pi_{\lambda}(X))^{-1}d\lambda\le \h^d(X')^{-2}\int_{\Lambda}||h_{\lambda}||_{L^2}^2d\lambda\le\frac{Bc^2\cdot 2^{2d+2k}}{w_k^2\h^d(X')^2}<\infty.$$
\end{proof}

\begin{rem}
The previous results also hold for maps $\pi:\Lambda\times X\to M$, where $M$ is a compact manifold (perhaps with boundary) of dimension $k$. Indeed, in this case, there are positive constants $\tilde a, \tilde b$ such that, for every $r\in (0,1)$, $\tilde a r^k\le vol(B(x,r))\le \tilde b r^k$, and the previous proofs work in the same way. 

If $M$ is an arbitrary manifold of dimension $k$, the previous theorems also hold, with the exception of the last statement of theorem \ref{thm_L2} (according to which $\int_{\Lambda}||h_{\lambda}||_{L^2}^2d\lambda<\infty$) and its corollary. Indeed, $M$ has an exhaustion $M=\cup_{n\ge 1}M_n$ as the union of compact manifolds $M_n$ with boundary such that $M_n\subset int M_{n+1}, \forall n\ge 1$. Assuming $X$ compact, since the maps $\pi_{\lambda}$ are continuous, for every $\lambda\in\Lambda$ there is $n\ge 1$ such that $\pi_{\lambda}(X)\subset M_n$. So, defining $\Lambda_n:=\{\lambda\in\Lambda; \pi_{\lambda}(X)\subset M_n\}$, we have $\cup_{n\ge 1}\Lambda_n=\Lambda$, so $\lim_{n\to\infty}\p[\Lambda_n]=1$, and our claim follows from the fact that the previous results hold for the restriction maps $\pi_n:\pi|_{\Lambda_n\times X}:\Lambda_n\times X \to M_n$ for every $n$ sufficiently large (for which $\p[\Lambda_n]>0$).
\end{rem}
\ \\
\ \\


\section{Some Applications}
\subsection{A Theorem by Solomyak} 
We will show how to use theorem \ref{thm_Marstrand2} to give a quick proof of a theorem by Solomyak (see \cite{Solomyak}):
 
\begin{thm}
Given $\beta\in (0,\frac{1}{2})$ then, for almost every $\lambda\in (0,1/2)$ such that $\frac{\log 2}{|\log \beta|}+\frac{\log2}{|\log \lambda|}>1$, the set $K_{\beta}+K_{\lambda}=\{x+y;x\in K_{\beta}, y\in K_{\lambda}\}$ has positive Lebesgue measure.
\end{thm}
\begin{proof}
Let $\Sigma=\{0,1\}^{\mathbb N}$. Given $\gamma\in (0,1/2)$, let $h_{\gamma}:\Sigma\to K_{\gamma}$ be the homeomorphism given by $h_{\gamma}(\theta)=(1-\gamma)\sum_{k=0}^{\infty}a_k \gamma^k$, where $\theta=(a_k)_{k\in \mathbb N}$. Let $\lambda\in (0,1/2)$ such that $\frac{\log 2}{|\log \beta|}+\frac{\log2}{|\log \lambda|}>1$. There are $\lambda_1, \lambda_2$ with $0<\lambda_1<\lambda<\lambda_2<1/2$ such that $\frac{\log \lambda_2}{\log \lambda_1}(\frac{\log 2}{|\log\beta|}+\frac{\log 2}{|\log\lambda_2|})>1$ (and in particular $\frac{\log\lambda_1}{\log\lambda_2}<2$). We will prove that, for almost every $\gamma\in [\lambda_1,\lambda_2]$, $K_{\beta}+K_{\gamma}$ has positive Lebesgue measure. The result then follows by considering a suitable countable union of such intervals $[\lambda_1,\lambda_2]$.

Take $N\in \mathbb N$ large so that $\frac1{N}<\frac1{2\lambda_2}-1$. Let $\Sigma_N=\{(a_k)_{k\in\mathbb N}\in\Sigma; a_k=0, \forall k<N\}\subset\Sigma$ and $X=\Sigma\times\Sigma_N$ endowed with the metric $d((\theta,\tilde\theta),(\tau,\tilde\tau))=|h_{\beta}(\theta)-h_{\beta}(\tau)|+|h_{\lambda_2}(\tilde\theta)-h_{\lambda_2}(\tilde\tau)|\in [0,2]$, and $\Lambda=[\lambda_1,\lambda_2]$, with the probability given by the normalized Lebesgue measure. Notice that $X$ is bilipschitz homeomorphic to $K_{\beta}\times (K_{\lambda_2}\cap [0,\lambda_2^N])$, and so has Hausdorff dimension $\frac{\log 2}{|\log\beta|}+\frac{\log 2}{|\log \lambda_2|}$. Given $\lambda\in\Lambda$, let $\pi_{\lambda}:X \to \mathbb R$ be given by $\pi_{\lambda}(\theta,\tilde\theta)=h_{\beta}(\theta)+h_{\lambda}(\tilde\theta)$. Notice that $\pi_{\lambda}(X)\subset K_{\beta}+K_{\lambda}$. Let $\alpha=\frac{\log \lambda_1}{\log\lambda_2}\in (1,2)$. Notice that we have $HD(X)>\alpha$.

Suppose that $(\theta,\tilde\theta)\ne (\tau,\tilde\tau)$ are two distinct elements of $X$. Let us estimate the probability that $|\pi_\lambda(\theta,\tilde\theta)-\pi_{\lambda}(\tau,\tilde\tau)|<\delta \cdot d((\theta,\tilde\theta),(\tau,\tilde\tau))^{\alpha}$, for a given $\delta\in (0,1/6)$ (for $\delta\in [1/6,1)$ the estimate in (\ref{transversality}) with $k=1$ holds automatically provided $C\ge 6$). We have $\pi_{\lambda}(\theta,\tilde\theta)-\pi_{\lambda}(\tau,\tilde\tau)=h_{\beta}(\theta)-h_{\beta}(\tau)+h_{\lambda}(\tilde\theta)-h_{\lambda}(\tilde\tau)$. Then we have
\begin{equation}\label{EQ0-Solom}
\frac{d}{d\lambda}(\pi_{\lambda}(\theta,\tilde\theta)-\pi_{\lambda}(\tau,\tilde\tau))=\frac{d}{d\lambda}(h_{\lambda}(\tilde\theta)-h_{\lambda}(\tilde\tau)).
\end{equation}
If $\tilde\theta=(a_k)_{k\in\mathbb N}$ and $\tilde\tau=(b_k)_{k\in\mathbb N}$, let $M=\min\{k\in \mathbb N;a_k\ne b_k\}\ge N$. Suppose without loss of generality that $a_M-b_M=1$. Then $h_{\lambda}(\tilde\theta)-h_{\lambda}(\tilde\tau)=(1-\lambda)(\lambda^M+\sum_{k=M+1}^{\infty} c_k \lambda^k)$, where $c_k\in \{-1,0,1\}, \forall k\ge M+1$. Since, for $k\ge M$, 
\begin{equation}\label{EQ1-Solom}
\frac{d}{d\lambda}((1-\lambda)\lambda^k)=k\lambda^{k-1}(1-\frac{k+1}{k}\lambda)\ge k\lambda^{k-1}(1-\frac{N+1}{N}\lambda)\ge k\lambda^{k-1}(1-\frac{\lambda}{2\lambda_2})\ge\frac12 k\lambda^{k-1}>0.
\end{equation}
Thus, from (\ref{EQ1-Solom}) we have
\begin{eqnarray}\label{EQ2-Solom}
\frac{d}{d\lambda}(h_{\lambda}(\tilde\theta)-h_{\lambda}(\tilde\tau))&=& M\lambda^{M-1}-(M+1)\lambda^M+\sum_{k=M+1}^{\infty}c_k\frac{d}{d\lambda}((1-\lambda)\lambda^k) \nonumber \\ 
&\geq &M\lambda^{M-1}-(M+1)\lambda^M -\sum_{k=M+1}^{\infty}\frac{d}{d\lambda}((1-\lambda)\lambda^k) \nonumber \\ 
&=& M\lambda^{M-1}-2(M+1)\lambda^M=M\lambda^{M-1}(1-2\cdot\frac{M+1}{M}\lambda)\nonumber \\  
&\ge & M\lambda^{M-1}(1-2\cdot\frac{N+1}{N}\lambda)\ge M\lambda_1^{M-1}(1-2\cdot\frac{N+1}{N}\lambda_2) \nonumber \\  
&\ge & C_1\lambda_1^{M-1}, 
\end{eqnarray}
 where $C_1:=M(1-2\cdot\frac{N+1}{N}\lambda_2)>0$. \\ 
 In particular, $|h_{\lambda}(\tilde\theta)-h_{\lambda}(\tilde\tau)|\le |h_{\lambda_2}(\tilde\theta)-h_{\lambda_2}(\tilde\tau)|, \forall \lambda\in [\lambda_1,\lambda_2]$. \\
If $|h_{\beta}(\theta)-h_{\beta}(\tau)|\ge 2|h_{\lambda_2}(\tilde\theta)-h_{\lambda_2}(\tilde\tau)|$ then, for every $\lambda\in [\lambda_1,\lambda_2]$, 
\begin{eqnarray*}
|\pi_{\lambda}(\theta,\tilde\theta)-\pi_{\lambda}(\tau,\tilde\tau)|&\ge & |h_{\beta}(\theta)-h_{\beta}(\tau)|-|h_{\lambda}(\tilde\theta)-h_{\lambda}(\tilde\tau)|\\ 
&\ge & |h_{\beta}(\theta)-h_{\beta}(\tau)|-|h_{\lambda_2}(\tilde\theta)-h_{\lambda_2}(\tilde\tau)|\\
&\ge & \frac13(|h_{\beta}(\theta)-h_{\beta}(\tau)|+|h_{\lambda_2}(\tilde\theta)-h_{\lambda_2}(\tilde\tau)|)\\
&=&\frac13 d((\theta,\tilde\theta),(\tau,\tilde\tau))\ge \delta\cdot d((\theta,\tilde\theta),(\tau,\tilde\tau))^{\alpha},
\end{eqnarray*}
since $\delta\in (0,1/6)$, $\alpha\in (1,2)$ and $d((\theta,\tilde\theta),(\tau,\tilde\tau))\in [0,2]$.\\
Suppose now that $|h_{\beta}(\theta)-h_{\beta}(\tau)|<2|h_{\lambda_2}(\tilde\theta)-h_{\lambda_2}(\tilde\tau)|$. Then, $$d((\theta,\tilde\theta),(\tau,\tilde\tau))=|h_{\beta}(\theta)-h_{\beta}(\tau)|+|h_{\lambda_2}(\tilde\theta)-h_{\lambda_2}(\tilde\tau)|<3|h_{\lambda_2}(\tilde\theta)-h_{\lambda_2}(\tilde\tau)|.$$
By definition of $h_{\lambda}$, we have 
\begin{eqnarray*}
h_{\lambda_2}(\tilde\theta)-h_{\lambda_2}(\tilde\tau)&\in &[(1-\lambda_2)(\lambda_2^M-\sum_{k=M+1}^{\infty} \lambda_2^k), \,(1-\lambda_2)(\lambda_2^M+\sum_{k=M+1}^{\infty} \lambda_2^k)]\\
&=&[(1-\lambda_2)\lambda_2^M-\lambda_2^{M+1},\lambda_2^M]\subset (0,\lambda_2^M], \, \, \text{since} \, \, \lambda_{2}<\frac{1}{2}.
\end{eqnarray*} On the other hand, consider the inequality
$|\pi_\lambda(\theta,\tilde\theta)-\pi_{\lambda}(\tau,\tilde\tau)|<\delta \cdot d((\theta,\tilde\theta),(\tau,\tilde\tau))^{\alpha}$.
Since, from (\ref{EQ0-Solom}) and (\ref{EQ2-Solom}), $\frac{d}{d\lambda}(\pi_{\lambda}(\theta,\tilde\theta)-\pi_{\lambda}(\tau,\tilde\tau)) \ge C_1\lambda_1^{M-1}$, by the mean value theorem, the set of values of $\lambda$ for which the above inequality holds is an interval whose measure is at most 
\begin{eqnarray*}
\frac{2\delta \cdot d((\theta,\tilde\theta),(\tau,\tilde\tau))^{\alpha}}{C_1\lambda_1^{M-1}}&<&\frac{2\delta (3|h_{\lambda_2}(\tilde\theta)-h_{\lambda_2}(\tilde\tau)|)^{\alpha}}{C_1\lambda_1^{M-1}} 
\le  \frac{2\delta(3\lambda_2^M)^{\alpha}}{C_1\lambda_1^{M-1}}  \\
&\le & \frac{18\delta(\lambda_2^M)^{\alpha}}{C_1\lambda_1^{M-1}}=\frac{18\lambda_1\delta}{C_1} \le C\delta,
\end{eqnarray*} with $C:=\max\{6,18\lambda_1/C_1\}$.  The result then follows directly from Theorem \ref{thm_Marstrand2}.
\end{proof}

\subsection{A result related to Falconer's Conjecture}
Falconer's conjecture (see \cite{Fal_conj}) states that for every subset $E\subset \mathbb{R}^d$, $d\geq 2$, such that $HD(E)>d/2$, then the distance set
\[
\Delta(E) = \{|x-y|;\ x\in E, y\in E\},
\]
has positive Lebesgue measure, where $|\cdot|$ represents the Euclidean norm in $\mathbb{R}^d$.\\
Given any point $\lambda$ of $\mathbb{R}^d$ we define the set 

$$\Delta_{\lambda}(E):=\{|\lambda-y|: y\in E\},$$ 
 of distances from the point $\lambda$ to all points in the set $E$.
We prove the following 
\begin{thm}\label{Falconer Conjecture}
If $HD(E)>1$, then for almost all $\lambda\in \mathbb{R}^d$ we have that $\Delta_{\lambda}(E)$ has positive Lebesgue measure.
\end{thm}

The last theorem guarantees the existence of a set  $\Gamma \in \mathbb{R}^d$ with full Lebesgue measure such that for every $x\in \Gamma$, the Lebesgue measure of the distance set $\Delta_{x}(E)$, is positive \emph{i.e.}, $\text{Leb}(\Delta_{x}(E))>0$. According with the Theorem \ref{Falconer Conjecture}, if $E\cap \Gamma \neq \emptyset$, then Falconer's conjecture is verified for the set $X$. 

This is a less stringent condition compared to Falconer's conjecture, which typically requires 
$HD(E)>\frac{d}{2}$. By ensuring $\Gamma$ exists under the condition $HD(E)>1$, it provides a broader scope for proving the conjecture.\\
To prove Theorem \ref{Falconer Conjecture}, we will leverage Theorem \ref{thm_Marstrand2} and employ some tools from linear algebra to ensure the transversality condition.\\
\indent Given $R>0$, $\delta>0$, $a\in \mathbb{R}$, and $v\in \mathbb{R}^d$, consider the set 
$$\mathcal{A}_{a, v, R}^{\delta}:=\Big\{w\in \mathbb{R}^d: |a-\langle v, w \rangle|\leq \frac{\delta}{2}\Big\}\cap B(0, R).$$

\begin{lem}\label{L-Falconer} There is a constant $\Theta$ (which only depends on $d$ and $R$) such that 
$$\text{Leb}(\mathcal{A}_{a, v, R}^{\delta})\leq \Theta \delta.$$
\end{lem} 
\begin{proof}
Note that $\Big\{w\in \mathbb{R}^d: |a-\langle v, w \rangle|= \frac{\delta}{2}\Big\}$ are two parallel hyperplane of dimension $d-1$ and normal vector $v$. From the symmetry of $B(0,R)$, for any $a\in \mathbb{R}$ we have 

$$\text{Leb}(\mathcal{A}_{a, v, R}^{\delta})\leq \text{Leb}(\mathcal{A}_{0, v, R}^{\delta}).$$
Using again the symmetry of the ball $B(0,R)$ we obtain that the Lebesgue measure of $\text{Leb}(\mathcal{A}_{0, v, R}^{\delta})$ is smaller than a cylinder of base $v^{\perp}\cap B(0,R)$ ($(d-1)$-dimensional ball) and height $\delta$. In other words, 
$$\text{Leb}(\mathcal{A}_{0, v, R}^{\delta})\leq \Theta\delta,$$
where $\Theta:= \dfrac{4^{d-1}(d-1)!}{(2(d-1))!}\pi^{\frac{d-1}{2}}R^{d-1}$ is the $(d-1)$-volume of a ball of radius $R$. 
\end{proof}

\begin{proof}[\emph{\textbf{Proof of Theorem \ref{Falconer Conjecture}}}]
Fix $R>0$ and consider the normalized Lebesgue measure on $B(0,R)$. Define $\pi\colon E\times B(0,R) \to \mathbb{R}$ as
$$\pi(x, \lambda)= ||x-\lambda||^2.$$
We will prove that $\pi$ satisfies the transversality condition (\ref{transversality}). For this sake, for every pair $x_1\neq x_2\in E $, we define the following set 
\begin{eqnarray*}
\mathcal{B}_{x_1, x_2, R}^{\delta}&:=&\Big\{\lambda\in \mathbb{R}^d: \Big |\pi(x_1,\lambda)-\pi(x_2,\lambda)\Big|\leq \delta||x_1-x_2||\Big\}\cap B(0,R)\\
&=&\Big\{\lambda\in \mathbb{R}^d: \Big|||x_1-\lambda||^2-||x_2-\lambda||^2\Big|\leq \delta||x_1-x_2||\Big\}\cap B(0,R)\\
&=&\Big\{\lambda\in \mathbb{R}^d: \Big| \frac{||x_1||^2-||x_2||^2}{||x_1-x_2||}-2\Big\langle \frac{x_1-x_2}{||x_1-x_2||}, \lambda \Big\rangle \Big |\leq \delta\Big\}\cap B(0,R)\\
&=& \mathcal{A}_{a, v, R}^{\delta},
\end{eqnarray*}
where $a= 2\dfrac{||x_1||^2-||x_2||^2}{||x_1-x_2||}$ and $v=\dfrac{x_1-x_2}{||x_1-x_2||}$.\\
So, from Lemma \ref{L-Falconer} $$\text{Leb}(\mathcal{B}_{x_1, x_2, R}^{\delta})=\text{Leb}(\mathcal{A}_{a, v, R}^{\delta})\leq \Theta \delta.$$
The last inequality provides the transversality condition (\ref{transversality}) with $\alpha = 1$ and $k=1$. Therefore, as $HD(E)>1=\alpha k$, it follows from Theorem \ref{thm_Marstrand2} that there is $\Gamma_R\subset B(0,R)$ of full Lebesgue measure such that $\text{Leb}(\pi_{\lambda}(E))>0$, for all $\lambda \in \Gamma_{R}$. Since the real function $x \to \sqrt{x}$ preserves positive measure we can conclude that 
$$\text{Leb}(\Delta_{\lambda}(E))>0, \,\,\, \text{for all} \,\,\, \lambda \in \Gamma,$$
where $\Gamma:= \displaystyle\bigcup_{R>0} \Gamma_{R}$, which  completes the proof of Theorem \ref{Falconer Conjecture}.

\end{proof}

\section{A version of Marstrand's theorem in non-positive curvature} \label{Example}
\subsection{{Geometric Marstrand}}\noindent 
Consider $\re^2$ with a complete Riemannian metric of non-positive curvature. Fix $p \in \re^{2}$ and let $\{e_1,e_2\}$ be a positive orthogonal basis of $T_p\re^2$, \emph{i.e.}, the basis $\{e_1,e_2\}$ has the induced orientation of $\re^{2}$. In these coordinates, we call $v_\lambda=(\cos \lambda,\sin \lambda):=(\cos \lambda) e_1+(\sin\lambda) e_2\in T_{p}\re^{2}$. Denote by $l_\lambda$ the line through $p$ with velocity $v_\lambda$, given by $l_\lambda(s)=exp_{p}sv_\lambda$, \emph{i.e.}, $l_\lambda(s)$ is a geodesic defined for all parameter values and minimizing the distance between any of its points. Since the curvature is non-positive, then orthogonal projection on $l_\lambda$ is well defined (cf. \cite{Bridson}). We denote by $\pi_\lambda$ such  projection. 
Then, we can define $\pi\colon [0,\pi)\times T_{p}\re^2 \to \re$ as, $\pi(\lambda, w):=\pi_{\lambda}(w)$, the unique parameter $s$ such that $\pi_{\lambda}(exp_{p}w)=\exp_{p}s{v_{\lambda}}$ \emph{i.e.}, 
$$\pi_{\lambda}(exp_{p}w)=exp_{p}\pi(\lambda, w)v_{\lambda}.$$

In this context, we prove the following theorem:

\begin{thm}[Geometric Marstrand]\label{thm_Marstrand4} Let ${\mathbb R}^2$ be endowed with a Riemannian metric of non-positive curvature, and $K \subset {\mathbb R}^2$ analytic. Then:
\begin{enumerate}
\item[\emph{1.}] If $HD(K)>1$ then, for almost every $\lambda \in [0 , \pi)$, we have that $m(\pi_\lambda(K))> 0$, where $m$ is the Lebesgue measure.

\item[\emph{2.}] If $HD(K)\le 1$ then, for almost every $\lambda \in [0 , \pi)$, we have that $HD(\pi_\lambda(K))=HD(K)$, where $\pi_l$ is the orthogonal projection on $l$.

\end{enumerate}
 
\end{thm}

From Hadamard's Theorem, any  Riemannian surface, simply connected, and with a Riemannian metric of non-positive curvature is diffeomorphic to $\re^2$ with metric of non-positive curvature. Therefore, the last theorem is also valid for Hadamard's surfaces. 

\subsection{Some Geometric Properties}
Let ${\displaystyle p\in \re^2}$. For $v\in T_p\re^2$, we make the identification $\displaystyle T_{v}T_{p}\re^2\cong T_{p}\re^2$.\\
\ \\
\noindent \textbf{Gauss's Lemma:} Let $v,w\in B_{\epsilon}(0)\in T_vT_p \re^2$ and $\re^2\ni q=exp_pv$. Then, 
$$\left\langle D(exp_p)_{v}v,D(exp_p)_{v}w\right\rangle_{q}=\left\langle v,w\right\rangle_{p},$$
where $exp_{p}$ is the exponential map. \\
\ \\
\noindent \textbf{The law of cosines:} Given three different $x, y, z \in \re^2$   
$$d^{2}(x,z)\geq d^{2}(y,x)+d^{2}(y,z)-2d(y,x)d(y,z)\cos \, \angle_{y}(x,z),$$
where $d(\cdot,\cdot)$ is the distance on $M$ generated by $\left\langle \, , \,\right\rangle$ and $\angle_{y}(x,z)$ is the angle in $y$ between the segments of geodesic $[y,x]$ and $[y,z]$. \\
Consequently, 
\begin{equation}\label{EL4}
d(x,y)\geq \sin \, \angle_{z}(x,y)\, \, d(x,z).
\end{equation}

\begin{rem}\label{R1}
\textit{Let $R>0$ be such $K\subset B(p,R)$ the ball of radius $R$ and center $p$, then there is a constant $c\geq 0$ such that 
\begin{equation}\label{EGMT1}
-c\leq \kappa(x)\leq 0 \ \ \text{for} \ \ x\in B(p,R),
\end{equation}
where $\kappa(x)$ is the Gaussian curvature at $x$.
Therefore, the equation \emph{(\ref{EGMT1})} implies that there is a constant $k\geq 1$ such that for each $x\in B(p, R)$ and $u,v\in T_x \re^2$ we have that  \emph{(cf.\, \cite{ManP})}
\begin{equation}\label{eq1}
k\left\|u-v\right\|\geq d(exp_x u, exp_x v)\geq \left\|u-v\right\|.
\end{equation}}
Here $\|v\|=\sqrt{\left\langle v , v \right\rangle}$.
\end{rem}
\ \\
\indent Given $u\neq v$ in ${\mathbb R}^2$, consider $\gamma_{uv}$ the unique geodesic such that  $\gamma_{uv}(0)=u$ and $\gamma_{u\,v}(d(u,v))=v$, then there are three functions  $\theta(u,v)\in [0,\pi)$,  $t(u,v), s(u,v)\in \re$  such that the geodesic $l_{\theta(u,v)}(s)=exp_{p}sv_{\theta(u,v)}$ intersects orthogonally the geodesic $\gamma_{uv}$ in the point 
$$I(u,v):=exp_{p}s(u,v)v_{\theta(u,v)}=\gamma_{uv}(t(u,v))$$ and the basis $\{l'_{\theta(u,v)}(s(u,v))=d(\,exp_{p})_{s(u,v)v_{\theta(u,v)}}(v_{\theta(u,v)}), \, \gamma_{uv}'(t(u,v))\}$  is positive. Note that if
 $\gamma_{uv}$ passes through $p$, then $\theta(u,v)$ is also well defined, and $s(u,v)=0$.\\
\noindent For $x\neq p$ denote by $\theta(x)\in [0,\pi)$ the angle such that the geodesic $exp_{p}tv_{\theta(x)}$ passes through $x$.

\subsection{Proof of Theorem \ref{thm_Marstrand4}}
The following is the main result of this section, it will be very useful in the  proof the Theorem \ref{thm_Marstrand4}.
\begin{lem}\label{L1'}
Let $K\subset \re^{2}$ a compact set, then there is a positive constant $\eta$ \emph{(}which only depends on $K$\emph{)} such that for any $u,v\in K$  
$$d(\pi_{\lambda+\theta(u,v)}(u), \pi_{\lambda+\theta(u,v)}(v))\geq \eta \sin (\lambda) \, d(u,v), \ \  \text{for any} \ \ \lambda\in[0,\pi].$$
\end{lem}
Using the last Lemma we are able to prove the transversality condition (\ref{transversality}) and consequently, we can use Theorem \ref{thm_Marstrand2}.
 \begin{proof}[\bf{Proof of Theorem \ref{thm_Marstrand4}}]
 By Corollary 7 in \cite{Howroyd_1995}, we may assume without loss of generality that $K$ is compact. Let $x_1, x_2\in \re^2$, without loss of generality we can assume $\theta(x_1,x_2)=0$. Thus, with the notation of Theorem \ref{thm_Marstrand2}, consider the probability space $([0,\pi),m)$ where $m$ is the normalized Lebesgue measure, $X=K$, $k=1$, and $\alpha=1$. We want to prove that $\pi(\lambda,x)$ satisfies the transversality property (\ref{transversality}). By Lemma \ref{L1'} we have that  $d(\pi_{\la}(x_1),\pi_{\la}(x_2))\geq \eta\sin(\la) d(x_1,x_2)$, thus 
\begin{eqnarray*}
m[\lambda\in [0,\pi):d(\pi_{\la}(x_1),\pi_{\la}(x_2))\leq\delta d(x_1,x_2)]&\leq& m[\lambda\in [0,\pi):\eta\sin \lambda\leq\delta ]\\
&=&m[\lambda\in [0,\pi):\sin \lambda\leq\delta/\eta].
\end{eqnarray*}
Since the function $\sin\lambda$ is $C^{1}$, then there is a constant $C>0$ such that for all $\delta>0$ 
$$m[\lambda\in [0,\pi):\sin \lambda\leq\delta/\eta]\leq C\delta,$$
so we have the transversality condition. Therefore the result follows directly from Theorem \ref{thm_Marstrand2}. 
\end{proof}

\subsubsection{Proof of Lemma \ref{L1'}}
\noindent  Before proving the Lemma \ref{L1'} we will give some definitions and we prove some auxiliary lemmas.\\
\indent Given $x\neq y$, then for $\epsilon>0$ small we can define the function $$\lambda \to \alpha_{y}(x,\lambda):= \angle_{x}(\pi_{\theta(x,y)+\lambda}(x), y) \ \ \text{and} \ \ \lambda \to \alpha_{x}(y,\lambda):= \angle_{y}(\pi_{\theta(x,y)+\lambda}(y), x) $$
for all $|\lambda|<\epsilon$. 
\begin{rem}The above definition is valid for  any $x\neq y$. However, for our goal, if $\gamma_{xy}$ passes through $p$, then we consider functions $\alpha_{p}(x,\lambda)$ and $\alpha_{p}(y,\lambda)$.
\end{rem}

We denote by $I(x,y,\lambda)$ the unique point of intersection of the geodesic $exp_{p}sv_{\theta(x,y)+\lambda}$ with the geodesic $\gamma_{xy}$ (cf. Figures \ref{Figure1}, \ref{Figure2}).
Note that, if the geodesic $\gamma_{xy}$  passes  through $p$, then $I(x,y,\lambda)=p$, for all $|\lambda|<\epsilon$. 



\begin{lem}\label{L2}
The functions $ (x,y)\to \pi_{\theta(x,y)+\lambda}(x)\, ; \, (x,y)\to  \pi_{\theta(x,y)+\lambda}(y)$ are differentiable. Moreover,  the function  $\lambda \to I(x,y,\lambda)$ is  differentiable in $\lambda$.
\end{lem}
\begin{proof}
The differentiability of $\pi_{\theta(x,y)+\lambda}(x)$ and $\pi_{\theta(x,y)+\lambda}(y)$ is an immediate consequence of the convexity of the distance function and Theorem of Implicit Functions.\\
To prove the second part, suppose that $\gamma_{xy}$ passes through $p$, then $I(x,y,\lambda)$ is constant equal to $p$, which is differentiable. Therefore, we can assume that $\gamma_{xy}$ does not pass through $p$, then there is a neighbourhood $U(x,y)\subset [0,2\pi)$ of $\theta(x,y)$ such that $I(x,y,\lambda)$ is a Poincar\'e-like map between $U(x,y)$ and $\gamma_{xy}(\re)$ defined as the intersection of the geodesic $\gamma_{v_{\lambda}}$ with $\gamma_{xy}$ for $\lambda\in U(x,y)$. Therefore, there is $\epsilon>0$ small such that  $(\theta(x,y)-\epsilon,\theta(x,y)+\epsilon )\subset U(x,y)$. Thus we concluded the proof of the second part.\ \\
\end{proof}
\begin{lem}\label{L3}
For $x,y\in \re^2$, $x\neq y$, the functions $\alpha_{y}(x,\lambda)$ and $\alpha_{x}(y,\lambda)$ are differentiable in $\lambda$. Moreover,
$$\lim_{\lambda\to 0} \frac{\alpha_{y}(x,\lambda)}{\lambda}=1 \ \text{and} \ \ \lim_{\lambda\to 0} \frac{\alpha_{x}(y,\lambda)}{\lambda}=1.$$


\end{lem}
\begin{proof}
The differentiability of both functions follows from Lemma \ref{L2} and the differentiability of the exponential map.  For the proof of limit, we consider two cases:\\
\noindent \textbf{Case 1:} The geodesic $\gamma_{xy}$ does not pass through $p$. \\
\noindent \textbf{Case 2:} The geodesic $\gamma_{xy}$ passes through $p$. 
 \ \\
We prove Case 1, since the proof of Case 2 is analogous.
\begin{figure}[htbp]
   \centering
    \begin{subfigure}[b]{0.43\textwidth}
        \includegraphics[width=\textwidth]{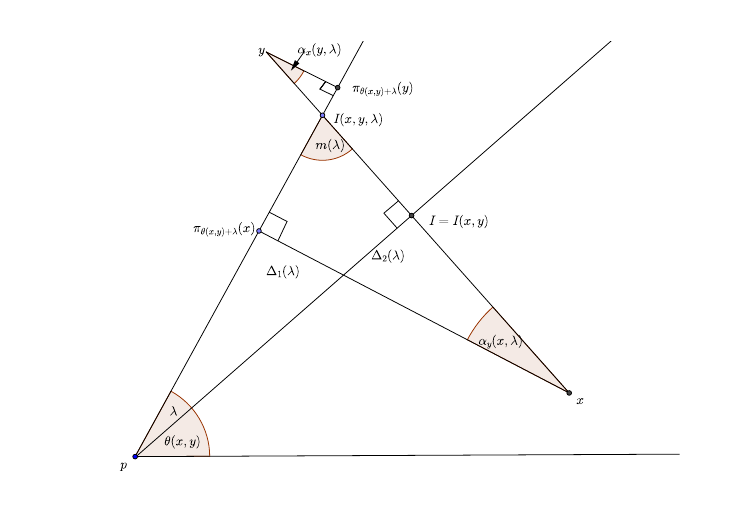}
        \caption{ }
        \label{Figure1}
    \end{subfigure}
   ~
   \begin{subfigure}[b]{0.53\textwidth}
        \includegraphics[width=\textwidth]{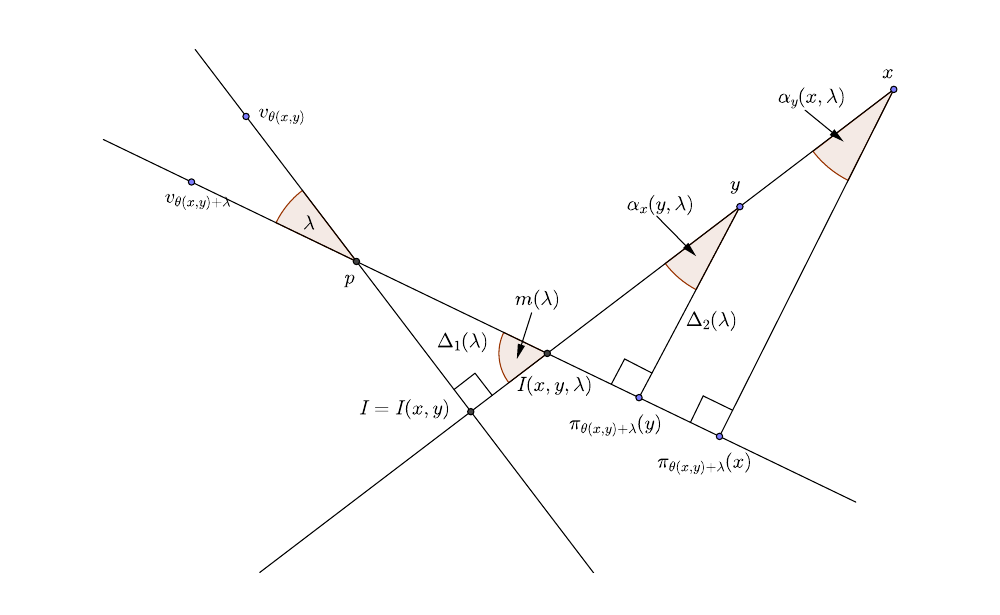}
        \caption{ }
       \label{Figure2}
    \end{subfigure}
    \caption{$\gamma_{xy}$ does not pass through $p$}\label{Figures1}
\end{figure} 

\noindent Put $I:=I(x,y,0)$ and consider the triangle $\Delta_{1}(\lambda)$ generated by the points $\{p,I,I(x,y,\lambda)\}$ and let $m(\lambda)=\angle_{I(x,y,\lambda)}(p,I)$ for $|\lambda|<\epsilon$ (see Figure 1a). So, by Lemma \ref{L2}, since $I(x,y,\lambda)$ is differentiable in $\lambda$, the function $m(\lambda)$ is differentiable in $\lambda$. Then, by the Gauss-Bonnet theorem, the function 
$$h_1(\lambda)=\int_{\Delta_{1}(\lambda)}\kappa(x)d\sigma(x)=-\pi/2+\lambda+m(\lambda)$$
where $\kappa(x)$ denotes the Gaussian curvature. 
Since $m(\lambda)$ is differentiable, then $h_1$ is also differentiable.\\ Moreover, since $\kappa(x)\leq 0$ then $h_1(\lambda)$ has a maximum in $\lambda=0$, thus  $h_1'(0)=0$.
So, we have that $0=h_1'(0)=m'(0)+ 1$ or $m'(0)=-1$.
Analogously, considering  the triangle $\Delta_{2}(\lambda)$ generated by the points $\{x,I(x,y,\lambda),\pi_{\theta(x,y)+\lambda}(x)\}$, we have that the function $h_{2}(\lambda):=\int_{\Delta_{2}(\lambda)}\kappa(x)d\sigma(x)=-\pi/2+m(\lambda)+ \alpha_{y}(x,\lambda)$ is differentiable with maximum in $\lambda=0$. Thus  $h_2'(0)=0$, which implies that  
$$0=h_2'(0)=m'(0)+\ds\lim_{\lambda\to 0} \frac{\alpha_{y}(x,\lambda)}{\lambda}.$$
Therefore, $\ds\lim_{\lambda\to 0} \frac{\alpha_{y}(x,\lambda)}{\lambda}=1$.
The proof for $\alpha_{x}(y,\lambda)$ is analogue.


\end{proof}

\begin{cor}\label{C2}
There is $\epsilon_1>0$ such that for all $x,y\in K$, $x\neq y $ and all $|\lambda|\leq \epsilon_1$, then
$$\sin \alpha_{y}(x,\lambda) \geq \frac{1}{2}\sin |\lambda| \ \ \text{and} \ \ \sin\alpha_{x}(y,\lambda)\geq \frac{1}{2}\sin |\lambda|.$$
\end{cor}

Let $B(p,R)$ be as in Remark \ref{R1}. Given three different points $x,y,z\in B(p, R)$, we consider the vectors $\tilde{z}_{x}=exp^{-1}_{x}(z)$, $\tilde{y}_{x}=exp^{-1}_{x}y$, $\tilde{x}_{z}=exp^{-1}_{z}(x)$ and $\tilde{y}_{z}=exp^{-1}_{z}y$. Then we have the following 
\begin{lem}\label{L5'} In the above notation, it holds that 
$$\sin\angle_{x}(y,z)\,\sin\angle_{z}(x,y)\geq \frac{1}{k^2}\sin \angle_{\tilde{x}_{z}}(0,\tilde{x}_y)\,\sin \angle_{\tilde{z}_{x}}(0,\tilde{z}_{y}),$$
where $k$ is the constant of Remark \emph{\ref{R1}}.
\end{lem}
\begin{proof}
Consider the triangle in $T_x\re^2$ generated by the points $0, \tilde{x}_{z}, \tilde{x}_{y}$ and the triangle in $T_z\re^2$ generated by the points $0, \tilde{z}_{x}, \tilde{z}_{y}$. The by law of sines

$$\frac{\sin\angle_{x}(y,z)}{\Vert\tilde{x}_{z}-\tilde{x}_{y}\Vert}=\frac{\sin \angle_{\tilde{x}_{z}}(0,\tilde{x}_y)}{d(x,y)} \ \ \text{and} \ \
\frac{\sin\angle_{z}(x,y)}{\Vert\tilde{z}_{x}- \tilde{z}_{y}\Vert}=\frac{\sin \angle_{\tilde{z}_{x}}(0,\tilde{z}_{y})}{d(z,y)}.$$
The conclusion followed by Remark \ref{R1}.
\end{proof}
\begin{lem}\label{L5}
Given three different points $x,y,z\in B(p,R)$ with $\angle_{y}(x,z)=\pi/2$, there is $\gamma>0$ such that 
$$\sin \angle_{\tilde{y}_{x}}(0,\tilde{z}_{x})\geq\gamma.$$
\end{lem}
\begin{proof}
Call $L$ the geodesic orthogonal to $\gamma_{xy}$ in $y$. Let $V(x,y)$ be such that $exp_{x}V(x,y)=y$ and put $v(x,y)=\frac{V(x,y)}{\Vert V(x,y)\Vert}$ and $v_{\lambda}(x,y)\in T_x\re^2$ the  vector that forms an angle $\lambda$ with $v(x,y)$. The function $f(\lambda)$ is defined as the unique point where the geodesic $exp_{x}sv_{\lambda}(x,y)$ meet the geodesic $L$. Thus, there exists a unique $s(\lambda)$ such that $f(\lambda)=exp_{x}s(\lambda)v_{\lambda}(x,y)$. It is easy to see that, following the same arguments of the proof of the second part of lemma \ref{L2}, the function $f$ is differentiable in $\lambda$. Consider the function $g(\lambda)=s(\lambda)v_{\lambda}(x,y)=exp_{x}^{-1}(f(\lambda))$ which is also differentiable. If $\theta(x,y,\lambda)$ denotes the angle between $g(0)=V(x,y)$ and $g(\lambda)-g(0)$.
The definition of $g(\lambda)$ and $\angle_{y}(x,z)=\pi/2$ implies that 
$$\langle d(exp_x)_{g(0)}g'(0),d(exp_x)_{g(0)}g(0) \rangle =\langle f'(0), d(exp_x)_{V(x,y)}V(x,y) \rangle=0.$$
Therefore, Gauss's Lemma provides that $\left\langle g'(0),g(0) \right\rangle=0$. Thus, we can conclude  $$\lim_{\lambda\to 0^{\pm}} \cos \theta(x,y,\lambda)=-\frac{\left\langle g(0),g'(0) \right\rangle}{\Vert g(0)\Vert \Vert g'(0)\Vert}=0.$$
Thus, since $x,y\in B(p,R)$, the above equation implies that there are $\tilde{\epsilon}>0$ and $\tilde{\gamma}>0$ independent of $x,y$ such that 
$$\sin \theta(x,y,\lambda)>\tilde{\gamma} \ \ \text{for} \ \ |\lambda|<\tilde{\epsilon} \ \ \text{and} \ \ x,y\in B(p, R).$$
Since $x,y,z\in B(p,R)$, $\angle_{y}(x,z)=\frac{\pi}{2}$ the definition of $g(\lambda)$ implies that there is a constant $c>0$ such that  $\dfrac{\Vert g(\lambda)\Vert}{\Vert g(\lambda)-g(0) \Vert}\geq c$. Therefore,  by the law of sines, if $|\lambda|\geq \tilde{\epsilon}$, 
$$\sin\theta(x,y,\lambda)=\frac{\Vert g(\lambda)\Vert}{\Vert g(\lambda)-g(0) \Vert}\sin |\lambda|\geq c\,\sin|\lambda|\geq c \sin \tilde{\epsilon}.$$
Taking $\gamma=\max\{c\,\sin \tilde{\epsilon}, \tilde{\gamma}\}$. As $z\in L$ the proof is complete.

\end{proof}
\begin{rem}\label{R2}
Observe that the previous Lemma holds if $|\angle_{y}(x,z)-\pi/2|<\delta$ for some $\delta\in (0,\frac{\pi}{2})$, since we only need that $|\cos\angle_{y}(x,z)|<a< 1$ for some $a>0$.
\end{rem}


\begin{proof}[\bf{Proof of Lemma \ref{L1'}}]
Suppose that $u\neq v$ and  the geodesic $\gamma_{uv}$ does not pass through $p$.
Without loss of generality we can assume that $\theta(u,v)=0$. Moreover, to simplify the notation, we put  $I(\lambda):=I(u,v,\lambda)$ and $I=I(u,v,0)$. Let us consider two cases:

\noindent $\textbf{Case 1:}$ $I\in [u,v]$ (cf. Figure \ref{Figure4}).
\begin{figure}[htbp]
	\centering
		\includegraphics[width=0.5\textwidth]{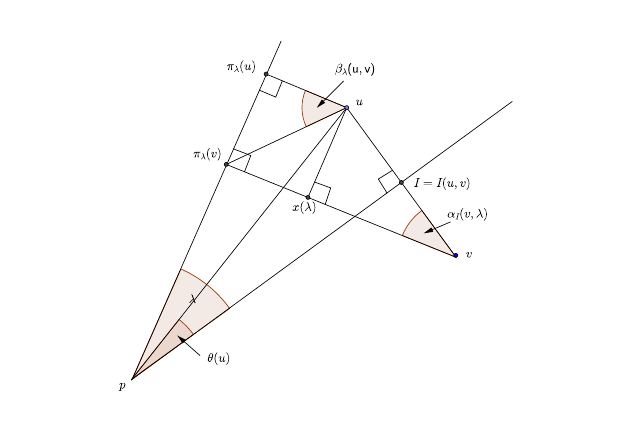}
	\caption{$I\in [u,v]$.}
	\label{Figure4}
\end{figure}
\noindent  From (\ref{EL4}), $$d(\pi_{\lambda}(v), I(\lambda))\geq \sin \alpha_{_{I}}(v,\lambda)\,d(I(\lambda), v), $$ 
$$d(\pi_{\lambda}(u), I(\lambda))\geq \sin \alpha_{_{I}}(u,\lambda)\,d(I(\lambda), u).$$
Let $\epsilon_1$ be given by Corollary \ref{C2}. Then, for $|\lambda|<\epsilon_1$, we have 

\begin{eqnarray}\label{eq2}
d(\pi_{\lambda}(u),\pi_{\lambda}(v))&=&d(\pi_{\lambda}(v), I(\lambda))+d(\pi_{\lambda}(u), I(\lambda)) \nonumber \\
&\geq & \max\{\sin \alpha_{I}(v,\lambda)\, , \sin \alpha_{I}(u,\lambda)\}d(u,v)\nonumber \\ 
 &\geq &\frac{1}{2}\sin |\lambda| \, d(u,v).
\end{eqnarray}

\noindent \textbf{Case 2:} $I\notin [u,v]$ (cf. figure \ref{Figure5}).\\
Without loss of generality, we can assume that for $\lambda$ small it holds that $d(u,\pi_{\lambda}(u))\geq d(v,\pi_{\lambda}(v))$. Denote $x(\lambda)=\pi_{[u,\pi_{\lambda}(u)]}(v)$, the projection of $v$ over the geodesic segment  $[u,\pi_{\lambda}(u)]$. Then, from (\ref{EL4}),  
 $$d(\pi_{\lambda}(u), \pi_{\lambda}(v))\geq \sin\, \beta(\lambda)d(\pi_{\lambda}(v),x(\lambda)),$$
 where $\beta(\lambda)=\angle_{x(\lambda)}(\pi_{\lambda}(u),\pi_{\lambda}(v))$.
 \begin{figure}[htbp]
	\centering
		\includegraphics[width=0.99\textwidth]{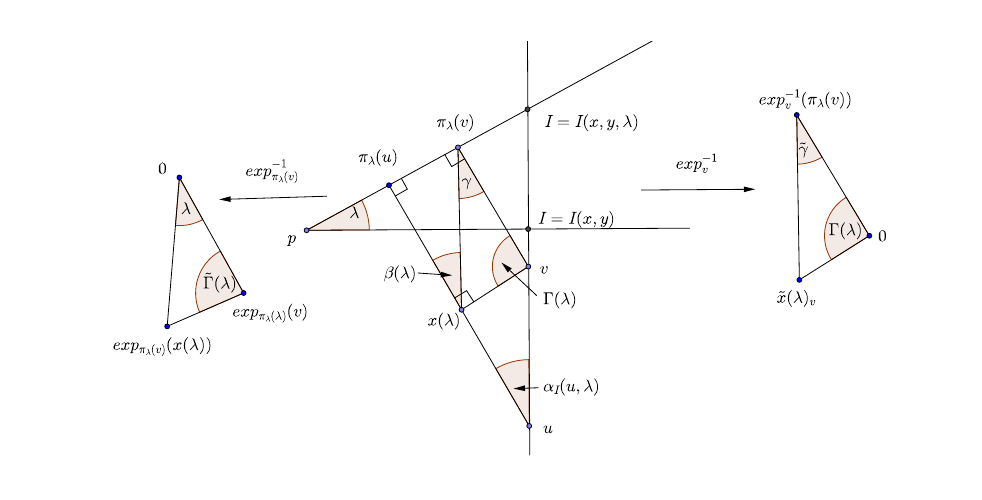}
	\caption{$I\notin [u,v]$.}
	\label{Figure5}
\end{figure}
\ \\
\noindent Considering the triangle in $T_{v}\re^2$ given by the points $0$, $\tilde{x}(\lambda)_{v}=exp_{v}^{-1}(x(\lambda))$ and 
$exp_{v}^{-1}(\pi_{\lambda}(v))$, then, by the law of sines,  $\sin\tilde{\gamma}(\lambda)=\frac{d(x(\lambda),v)\sin \, \Gamma(\lambda)}{\Vert\tilde{x}(\lambda)_{v}-exp_{v}^{-1}(\pi_{\lambda}(v))\Vert}$, where $\tilde{\gamma}(\lambda)=\angle_{exp_{v}^{-1}(\pi_{\lambda}(v))}(0, \tilde{x}(\lambda)_v)$ and $\Gamma(\lambda)=\angle_{v}(\pi_{\lambda}(v),x(\lambda))$. \\
Since  $\Vert\tilde{x}(\lambda)_{v}-exp_{v}^{-1}(\pi_{\lambda}(v))\Vert\leq d(\pi_{\lambda}(v),x(\lambda))$ we have that 
$$d(\pi_{\lambda}(u), \pi_{\lambda}(v))\geq \frac{\sin\, \beta(\lambda)}{\sin \tilde{\gamma}(\lambda)}\sin\, \Gamma(\lambda)\,d(x(\lambda),v).$$
Using (\ref{EL4}) and Corollary \ref{C2} we have that 
$$d(\pi_{\lambda}(u), \pi_{\lambda}(v))\geq \frac{1}{2}\frac{\sin\, \beta(\lambda)}{\sin \tilde{\gamma}(\lambda)}\sin\, \Gamma(\lambda)\,\sin|\lambda|\,d(u,v)\ \ \text{for} \ \ |\lambda|\leq\epsilon_1.$$
Now, since $\angle_{x(\lambda)}(v, \pi_{\lambda}(u))=\angle_{\pi_{\lambda}(u)}(x(\lambda), \pi_{\lambda}(v))=\angle_{\pi_{\lambda}(v)}(v, \pi_{\lambda}(u))=\pi/2$, remark \ref{R1} and Gauss Bonnet Theorem imply that $\Gamma(\lambda)\to\pi/2$ as $\lambda\to 0$. Thus, we can assume that there is a constant $c_1$ such that 
\begin{equation}\label{eq11}
d(\pi_{\lambda}(u), \pi_{\lambda}(v))\geq \frac{c_1}{2}\frac{\sin\, \beta(\lambda)}{\sin \tilde{\gamma}(\lambda)}\,\sin|\lambda|\,d(u,v)\ \ \text{ for } \ \ |\lambda|\leq \epsilon_1.
\end{equation}
 
Considering the triangle in $T_{\pi_{\lambda}(v)}\re^2$ given by the points $0$, $\exp_{\pi_{\lambda}(v)}^{-1}(x(\lambda))$ and 
$\exp_{\pi_{\lambda}(v)}^{-1}(v)$, if $\tilde{\Gamma}(\lambda)=\ds\angle_{\exp_{\pi_{\lambda}(v)}^{-1}(v)}\,(0,\exp_{\pi_{\lambda}(v)}^{-1}(x(\lambda)))$, Lemma \ref{L5'} implies that 
 $$\frac{\sin \gamma(\lambda)}{\sin \tilde{\gamma}(\lambda)}\geq \frac{1}{k^2}\cdot \frac{\sin \tilde{\Gamma}(\lambda)}{\sin\Gamma(\lambda)},$$
where $\gamma(\lambda)=\angle_{\pi_{\lambda}(v)}(v,x(\lambda))$.
Thus, since $\Gamma(\lambda) \to \frac{\pi}{2}$ as $\lambda \to 0$, the Remark \ref{R2} of Lemma \ref{L5} implies that there is a constant $c_2$ such that $\frac{\sin \tilde{\Gamma}(\lambda)}{\sin\Gamma(\lambda)}\geq c_2$ for  $|\lambda|\leq \epsilon_1$, which provides that 
\begin{equation}\label{eq11'}
\frac{\sin \gamma(\lambda)}{\sin \tilde{\gamma}(\lambda)}\geq \frac{c_2}{k^2},\ \ \ \ \text{ for }\,|\lambda|\leq \epsilon_1.
\end{equation}
Moreover, as the curvature is non-positive, then $\gamma(\lambda)-(\frac{\pi}{2}-\Gamma(\lambda))\leq \beta(\lambda)$, thus, as \textit{sine} is an increasing function and $\Gamma(\lambda)\to 0$, then without loss of generality,  we have $\sin{\beta(\lambda)}\geq \frac{\sin \gamma(\lambda)}{2}$ for $|\lambda|<\epsilon_1$.
This last inequality together with the equations (\ref{eq11}) and (\ref{eq11'}) implies that 
\begin{equation}\label{eq12}
d(\pi_{\lambda}(u), \pi_{\lambda}(v))\geq \frac{c_1c_2}{4k^2}\sin|\lambda|\,d(u,v),\ \  \ \ |\lambda|\leq \epsilon_1.
\end{equation}
Taking $\tilde{\eta}=\max\{\frac{c_1c_2}{4k^2} ,\frac{1}{2}\}$, then equations (\ref{eq2}) and (\ref{eq12}) allow us to conclude that  
\begin{equation}\label{eq123}
d(\pi_{\lambda}(u), \pi_{\lambda}(v))\geq \tilde{\eta}\sin|\lambda|\,d(u,v), \ \  \ \ |\lambda|<\epsilon_1.
\end{equation}
From Corollary \ref{C2}, the number $\epsilon_1$ does not depend on $u$ and $v$, so the equation (\ref{eq123}) is valid for all $u,v \in K$.\\
Moreover, by definition of $\theta(\cdot,\cdot)$ it is easy to see that there is a constant $c_3>0$ such that 
\begin{equation*}\label{eq124}
d(\pi_{\lambda}(u), \pi_{\lambda}(v))\geq c_3\sin|\lambda|\,d(u,v), \ \  \,\text{ for } \, u,v \in K, \, \, \, \, |\lambda|\geq \epsilon_1.
\end{equation*}
So, we take $\eta:=\max\{\tilde{\eta},c_3\}$. To finish the proof of the lemma, we note that if the geodesic $\gamma_{uv}$ passes through $p$, then we proceed  similarly to  case 2.
\end{proof}
\bibliographystyle{alpha}	
\bibliography{bibtex}

\noindent

\end{document}